\documentclass{amsart}

\usepackage{latexsym,amsthm,mathrsfs,amsmath,amscd,enumerate,enumitem, amscd,color,dsfont, cancel, mathtools, textcomp, float}

\usepackage[utf8x]{inputenc}

\usepackage[cspex,bbgreekl]{mathbbol}
\usepackage{amsfonts}
\usepackage{amssymb}             

\DeclareSymbolFontAlphabet{\mathbbl}{bbold}
\DeclareSymbolFontAlphabet{\mathbb}{AMSb}%

 \newtheorem{thm}{Theorem}[section]
 
 \newtheorem{lem}[thm]{Lemma}
 
\theoremstyle{definition}

 \theoremstyle{remark}
 \newtheorem{rem}[thm]{Remark}

\usepackage{hyperref}
\usepackage{tikz}
\usetikzlibrary{plotmarks}

\makeatletter
\DeclareRobustCommand\widecheck[1]{{\mathpalette\@widecheck{#1}}}
\def\@widecheck#1#2{%
    \setbox\z@\hbox{\m@th$#1#2$}%
    \setbox\tw@\hbox{\m@th$#1%
       \widehat{%
          \vrule\@width\z@\@height\ht\z@
          \vrule\@height\z@\@width\wd\z@}$}%
    \dp\tw@-\ht\z@
    \@tempdima\ht\z@ \advance\@tempdima2\ht\tw@ \divide\@tempdima\thr@@
    \setbox\tw@\hbox{%
       \raise\@tempdima\hbox{\scalebox{1}[-1]{\lower\@tempdima\box
\tw@}}}%
    {\ooalign{\box\tw@ \cr \box\z@}}}
\makeatother

\usepackage{caption}
\usepackage{subcaption}

\usepackage[left=2.2cm,top=2.5cm,right=2.2cm,bottom=2.5cm]{geometry} 

\numberwithin{equation}{section}
\allowdisplaybreaks

\begin{document}

\title[]{Maximal operator, Littlewood-Paley functions and variation operators associated with nonsymmetric Ornstein-Uhlenbeck operators.}

\author[V. Almeida]{V\'{\i}ctor Almeida}

\author[J.J. Betancor]{Jorge J. Betancor}

\author[P. Quijano]{Pablo Quijano}

\author[L. Rodr\'{\i}guez-Mesa]{Lourdes Rodr\'{\i}guez-Mesa}

\address{V\'{\i}ctor Almeida, Jorge J. Betancor, Lourdes Rodr\'{\i}guez-Mesa\newline
	Departamento de An\'alisis Matem\'atico, Universidad de La Laguna,\newline
	Campus de Anchieta, Avda. Astrof\'isico S\'anchez, s/n,\newline
	38721 La Laguna (Sta. Cruz de Tenerife), Spain}
\email{valmeida@ull.edu.es, jbetanco@ull.es,
lrguez@ull.edu.es
}

\address{Pablo Quijano\newline
	Instituto de Matemática Aplicada del Litora, Santa Fe-Argentina}
\email{pabloquijanoar@gmail.com}

\thanks{The authors are partially supported by grant PID2019-106093GB-I00 from the Spanish Government}



\date{}

\begin{abstract}
In this paper we establish $L^p$ boundedness properties for maximal operators, Littlewood-Paley functions and variation operators involving Poisson semigroups and resolvent operators associated with nonsymmetric Ornstein-Uhlenbeck operators. We consider the Ornstein-Uhlenbeck operators defined by the identity as the covariance matrix and having a drift given by the matrix $-\lambda(I+R)$, being $\lambda>0$ and $R$ a skew-adjoint matrix. The semigroup associated with these Ornstein-Uhlenbeck operators are the basic building blocks of all the normal Ornstein-Uhlenbeck semigroups.
\end{abstract}

\maketitle

\section{Introduction}

In this paper we are concerned with maximal operators, Littlewood-Paley functions and variation operators defined by Poisson semigroups and resolvent operators for nonsymmetric Ornstein-Uhlenbeck operators.

We denote by $Q$ a real, symmetric and positive definite $d\times d$ matrix and by $B$ a nonzero real $d\times d$ matrix having eigenvalues with negative real parts, being $d\in \mathbb{N}$, $d\geq 1$. We now introduce the Ornstein-Uhlenbeck semigroup defined by $Q$, named the covariance matrix, and $B$ called the drift matrix. For every $t\in(0,\infty]$ we consider the matrix $Q_t$ given by
\begin{equation*}
    Q_t = \int_0^t e^{sB} Q e^{sB^*}ds,
\end{equation*}
and the Gaussian measure $\gamma_t$ on $\mathbb R^d$ having mean zero and covariance matrix $Q_t$ defined by
\begin{equation*}
    d\gamma_t(x) = (2\pi)^{-\frac{d}{2}}(\det Q_t)^{-\frac{1}{2}}
    e^{-\frac{1}{2}\langle Q^{-1}_t x, x\rangle} dx.
\end{equation*}

The Ornstein-Uhlenbeck semigroup is $\{\mathcal{H}^{Q,B}_t\}_{t>0}$, where, for every $t>0$, 
\begin{equation}\label{1.1}
    \mathcal{H}^{Q,B}_t(f)(x) = \int_{\mathbb R^d} f\left( 
    e^{tB}x-y\right) d\gamma_t(y), \;\;\; x\in\mathbb R^d,
\end{equation}
where $f$ belongs to the space $\mathcal{C}_b(\mathbb R^d)$ of bounded continuous functions in $\mathbb R^d$.

The semigroup $\{\mathcal{H}^{Q,B}_t\}_{t > 0}$ is the transition semigroup of the Ornstein-Uhlenbeck process on $\mathbb R^d$ (\cite{DaPZ}). The measure $\gamma_\infty$ is the unique invariant measure for $\mathcal{H}_t^{Q,B}$, $t>0$. Furthermore, the equality~\eqref{1.1} defines a semigroup of positive contractions in $L^p(\mathbb R^d,\gamma_\infty)$, for every $1\leq p <\infty$.

The Ornstein-Uhlenbeck operator $\mathcal{L}^{Q,B}$ is defined by
\begin{equation*}
    \mathcal{L}^{Q,B} f(x) = \frac{1}{2} \text{tr}(Q \nabla^2 f)(x) +
    \langle Bx, \nabla f(x)\rangle, \;\;\; x\in\mathbb R^d,
\end{equation*}
when $f\in \mathcal{C}^2(\mathbb R^d)$, the set of differentiable functions with continuity up to order two . Here, $\nabla$ denotes the gradient and $\nabla^2$ represents the Hessian. Thus, $-\mathcal{L}^{Q,B}$ is an elliptic operator. If $1\leq p < \infty$, by naming $-\mathcal{L}^{Q,B}_p$ the infinitesimal generator of $\{\mathcal{H}^{Q,B}_t\}_{t > 0}$ in $L^p(\mathbb R^d,\gamma_\infty)$, we have that $\mathcal{L}^{Q,B}_p f = \mathcal{L}^{Q,B}f$, $f\in S(\mathbb R^d)$, the set of Schwartz functions,  and $S(\mathbb R^d)$ is dense in the domain $\mathcal{D}(\mathcal{L}_p^{Q,B})$ of $\mathcal{L}_p^{Q,B}$. In~\cite{MPRS}, it was proved that $\mathcal{D}(\mathcal{L}_p^{Q,B})$ coincides with the Sobolev space $W^{2,p}(\gamma_\infty)$.

Harmonic analysis associated with the symmetric Ornstein-Uhlenbeck operator has been much investigated over the last twenty five years. When $Q = -B = I$, where $I$ denotes the identity matrix, the operator $\mathcal{L}$ reduces to the classical Ornstein-Uhlenbeck operator $\mathcal{L}^{I,-I} = \frac{1}{2} \Delta - x \nabla$ and the Hermite polynomials are an orthonormal basis in $L^2(\mathbb R^d, \gamma_\infty)$ of eigenfunctions of $\mathcal{L}^{I,-I}$. Muckenhoupt (\cite{Mu1}) studied maximal operator and Riesz transforms in the one dimensional $\mathcal{L}^{I,-I}$-setting . Sj\"ogren (\cite{Sj}) extended to higher dimensions Muckenhoupt's results about the maximal operator defined by $\{\mathcal{H}^{I,-I}_t\}_{t>0}$. Harmonic analysis operators associated with $\mathcal{L}^{I,-I}$ were studied in~\cite{GCMMST2} and~\cite{MPS} (maximal operators); in~\cite{FSU},~\cite{P2} and~\cite{PS} (Littlewood-Paley functions); in~\cite{FHS}, \cite{GCMST1}, \cite{P1} and~\cite{PS} (Riesz transforms); in~\cite{GCMMST1} and~\cite{GCMST2} (spectral multipliers) and in~\cite{HMMT} (variation and oscillation operators). Guti\'errez, Segovia and Torrea (\cite{GST}) and Guti\'errez (\cite{G}) studied Riesz transforms defined by the operator $\mathcal{L}^{I,B}$ when $B$ is symmetric.

Mauceri and Noselli proved $L^p$ boundedness properties for maximal operators (\cite{MN1}) and Riesz transforms (\cite{MN2}) when $Q = I$ and $B=-\lambda(I+R)$, with $\lambda >0$ and $R$ is a skew-adjoint matrix. The semigroups $\{\mathcal{H}^{I,-\lambda(I+R)}_t\}_{t>0}$ are the basic building blocks of normal Ornstein-Uhlenbeck semigroups because, after a change of variables, any normal Ornstein-Uhlenbeck semigroup can be written as a product of commuting semigroups of that form.

Recently, Casarino, Ciatti and Sj\"ogren (\cite{CCS2}, \cite{CCS1} and \cite{CCS3}) have extended the results about maximal operators and Riesz transforms due to Mauceri and Noselli (\cite{MN2} and~\cite{MN1}).

Our objective in this paper is to establish $L^p$ boundedness properties of some maximal operators, Littlewood-Paley functions and variation operators involving the Poisson semigroups and the resolvent operators associated with the nonsymmetric Ornstein-Uhlenbeck operator considered by Mauceri and Noselli  (\cite{MN2} and~\cite{MN1}).

Assume that $Q = I$ and $B = -\lambda(I + R)$ where $\lambda > 0$ and $R$ is a skew-adjoint matrix as in~\cite{MN2} and~\cite{MN1}. After making a change of variables in~\eqref{1.1} we get
\begin{equation*}
    \mathcal{H}^{Q,B}_t (f)(x) = 
    \int_{\mathbb{R}^d} \widetilde{h}^{Q,B}_t (x,y) f(y) d\gamma_{\infty,\lambda} (y), \quad  x\in\mathbb{R}^d\text{ and }t>0,
\end{equation*}
where $d\gamma_{\infty,\lambda} (y) = 
\left(\frac{\pi}{\lambda }\right)^{-d/2} e^{-\lambda |y|^2}dy$, and
\begin{equation*}
     \widetilde{h}^{Q,B}_t (x,y) = \det(Q_\infty Q^{-1}_t)^{\frac{1}{2}}
    e^{-\frac{1}{2} \left[ 
    \langle Q^{-1}_t (e^{tB}x - y), e^{tB}x - y\rangle 
    - \langle Q^{-1}_\infty y, y\rangle
    \right] },
\end{equation*}
for $x$, $y\in \mathbb{R}^d$ and $t>0$. 

By using the subordination formula, the Poisson semigroup $\{P^{Q,B}_t\}_{t>0}$ is given by
\begin{equation*}
    P^{Q,B}_t(f) = \frac{t}{2\sqrt{\pi}} \int_0^{\infty} \frac{e^{-\frac{t^2}{4u}}}{u^{\frac{3}{2}}} \mathcal{H}^{Q,B}_u (f)du, \quad t>0.
\end{equation*}
Let $k\in\mathbb{N}$ and $j=1,\dots,d.$ We consider the maximal operator $P^{Q,B}_{*,k,j}$ defined by
\begin{equation*}
    P^{Q,B}_{*,k,j} (f) = 
    \sup_{t>0} |t^{k+1} \partial^k_t \partial_{x_j} P^{Q,B}_t (f)|.
\end{equation*}

The Littlewood-Paley $g^{Q,B}_{k,j}$ is given by
\begin{equation*}
    g^{Q,B}_{k,j} (f) = 
    \left( \int_0^\infty 
    \left| t^{k+1} \partial^k_t \partial_{x_j} P^{Q,B}_t (f)\right|^2 \frac{dt}{t} \right)^{\frac{1}{2}}.
\end{equation*}

Let $\rho>2$. If $g$ is a complex valued function defined in $(0,\infty)$, the $\rho$-variation $V_\rho(g)$ of $g$ is defined by
\begin{equation}\label{1.2}
    V_\rho(g) = \sup_{\substack{ 0<t_\ell<t_{\ell-1}<t_1\\ \ell \in \mathbb{N}}} \left( \sum_{n=1}^{\ell-1} 
    |g(t_{n+1}) - g(t_n)|^\rho \right)^{\frac{1}{\rho}}.
\end{equation}

Variation inequalities have been studied in probability, ergodic theory and harmonic analysis in recent years, The first variation inequality was due to L\'eplinge (\cite{le}) in the martingales setting. Later, Bourgain (\cite{Bo}) studied variation operators associated with ergodic averages of dynamic systems. The last paper has motivated a lot of researches in ergodic theory and harmonic analysis. We recommend to the interested reader the following recent papers and the reference therein: \cite{BMSW}, \cite{MTX1}, \cite{MTX2}, \cite{MSZ}, \cite{MTZ} and \cite{Z}.

We consider the variation operator $V^{Q,B}_{\rho,k,j}$ given by
\begin{equation*}
    V^{Q,B}_{\rho,k,j}(f)(x) = 
    V_\rho \left( 
    t^{k+1} \partial^k_t \partial_{x_j} P^{Q,B}_t (f)(x)
    \right),\quad x\in\mathbb{R}^d.
\end{equation*}

Note that, for every $f\in L^p(\mathbb R^d,\gamma_{\infty,1})$, $1\leq p<\infty$, the function $ V^{Q,B}_{\rho,k,j}(f)$ is a Lebesgue measurable  function because, for almost everywhere $x\in\mathbb R^d$, the function $F_x(t)=t^{k+1}\partial_t^k\partial_{x_j}P_t^{Q,B}(f)(x)$, $t\in(0,\infty)$, is continuous (see the comments after \cite[Theorem 1.2]{CJRW}). This measurability  property also holds for the other variation operators considered in this paper.

$g$-Littlewood-Paley functions associated with the symmetric Ornstein-Uhlenbeck operator $\mathcal{L}^{I,-I}$ were studied in~\cite{FSU},~\cite{P2} and~\cite{PS}. The $\rho$-variation operator for the Poisson semigroup $\{ P^{I,-I}_t\}$ defined by the symmetric Ornstein-Uhlenbeck operator without any derivatives was considered in~\cite{HMMT}.

Let $k\in\mathbb{N}$, $j=1,\dots,d$, and $M\geq 1$. We consider the following maximal operators, Littlewood-Paley functions and variation operators involving resolvent operators of $\mathcal{L}^{Q,B}$ defined by
\begin{equation*}
    S^{Q,B}_{*,k,j,M} (f) = 
    \sup_{t>0} |t^{k+\frac{1}{2}} \partial^k_t \partial_{x_j} (I + t\mathcal{L}^{Q,B})^{-M}(f)|,
\end{equation*}

\begin{equation*}
    G^{Q,B}_{k,j,M} (f) = 
    \left( \int_0^\infty 
    \left| t^{k+\frac{1}{2}} \partial^k_t \partial_{x_j} (I + t\mathcal{L}^{Q,B})^{-M} (f)\right|^2 \frac{dt}{t} \right)^{\frac{1}{2}}, 
\end{equation*}
and
\begin{equation*}
    \mathcal{V}^{Q,B}_{\rho,k,j,M}(f)(x) = 
    V_\rho \left( 
    t^{k+\frac{1}{2}} \partial^k_t \partial_{x_j} (I + t\mathcal{L}^{Q,B})^{-M}(f)(x)
    \right),\quad x\in\mathbb{R}^d.
\end{equation*}
Square functions of $G$-type have been recently defined in other settings (\cite{FMcP}). 

Our main result is the following one.

\begin{thm}\label{Th1.1}
Let $k\in\mathbb{N}$, $M>(d+1)/2$, $j=1,\dots,d$ and $\rho>2$.
Assume that $Q = I$ and $B = -\lambda(I+R)$, where $\lambda>0$ and $R$ is a skew-adjoint matrix that generates a periodic one parameter group $\{e^{tR}\}_{t\in\mathbb{R}}$. Then, the operators $P^{Q,B}_{*,k,j}$, $g^{Q,B}_{k,j}$, $V^{Q,B}_{\rho,k,j}$, $S^{Q,B}_{*,k,j,M}$, $G^{Q,B}_{k,j,M}$ and $\mathcal{V}^{Q,B}_{\rho,k,j,M}$ are bounded from $L^p(\mathbb{R}^d, \gamma_{\infty,\lambda})$ into $L^p(\mathbb{R}^d, \gamma_{\infty,\lambda})$ and from $L^1(\mathbb{R}^d, \gamma_{\infty,\lambda})$ into $L^{1,\infty}(\mathbb{R}^d, \gamma_{\infty,\lambda})$.
\end{thm}

As it is usual since~\cite{Mu1} was published, in the study of $L^p$-boundedness properties of harmonic analysis operators in the Ornstein-Uhlenbeck setting, in order to prove Theorem~\ref{Th1.1} we decompose the operator in two ones that are named the local part and the global part of the operator under consideration. The local part, as the original operator, is a singular integral, while the global part is controlled by a positive operator. 

In Section 2 we give the definitions and known results   that will be useful  in the sequel. We also explain the method used in the proof of Theorem \ref{Th1.1}. $L^p$-boundedness properties stated in Theorem \ref{Th1.1} are proved in Sections 3 and 4.

Throughout this paper $C$ and $c$ always represent positive constants that can change in each occurrence.

\section{Preliminaries}

Assume that $Q=I$ and $B=-\lambda(I+R)$, where $\lambda >0$ and $R$ is skew-adjoint. In this case we have that
\begin{align*}
d\gamma_\infty(x):=d\gamma_{\infty,\lambda}(x)=\left(\frac{\pi}{\lambda}\right)^{-d/2}e^{-\lambda|x|^2}dx.
\end{align*}
Actually we are going to work with $\lambda=1$. We define $B_1=-(I+R)$. We have that
\begin{align*}
\mathcal H_t^{Q,B}=\mathcal U_\lambda^{-1}\mathcal H_{\lambda t}^{Q,B_1}\mathcal U_\lambda,\;\;\;\;t>0,
\end{align*}
where $\mathcal U_\lambda(f)(x)=f(x/\sqrt{\lambda})$, $x\in\mathbb R^d$. It is clear that $\mathcal U_\lambda$ is an isometry from $L^p(\mathbb R^d,\gamma_{\infty,\lambda})$ (respectively, $L^{p,\infty}(\mathbb R^d,\gamma_{\infty,\lambda})$) into $L^p(\mathbb R^d,\gamma_{\infty,1})$ (respectively, $L^{p,\infty}(\mathbb R^d,\gamma_{\infty,1})$), for every $1\leq p <\infty$.

After a change of variable we can write
\begin{align}\label{2.1}
P_t^{Q,B}(f)(x) & = \frac{t}{2\sqrt{\pi}}\int_0^\infty\frac{e^{-\frac{t^2}{4u}}}{u^{\frac{3}{2}}}U_\lambda ^{-1}\mathcal H_{\lambda u}^{Q,B_1}(U_\lambda f)(x)du \nonumber\\
& =\mathcal U_\lambda^{-1}\left[\frac{\sqrt{\lambda}t}{2\sqrt{\pi}}\int_0^\infty\frac{e^{-\frac{\lambda t^2}{4v}}}{v^{\frac{3}{2}}}\mathcal H_{v}^{Q,B_1}(U_\lambda(f))(\cdot)dv\right](x) \nonumber \\
& =\mathcal U_\lambda^{-1}\left[P_{\sqrt{\lambda}t}^{Q,B_1}(U_\lambda(f))\right](x),\;\;\;\;x\in\mathbb R^d\;\mbox{and}\;t>0.
\end{align}
By using (\ref{2.1}) we deduce that, by denoting $T^{Q,B}$ every of the operators considered in Theorem \ref{Th1.1}, the following equality holds
\begin{align*}
T^{Q,B}=\mathcal U_\lambda^{-1}T^{Q,B_1}\mathcal U_\lambda.
\end{align*}
Thus, we show that it is sufficient to prove Theorem \ref{Th1.1} when $\lambda =1$. In the sequel we assume $\lambda =1$.

We can write
\begin{align*}
\mathcal H_t^{I,B_1}(f)(x)=\int_{\mathbb R^d}h_t^{I,B_1}(x,y)f(y)dy,\quad x\in\mathbb R^d\mbox{ and }t>0,
\end{align*}
where
\begin{align*}
h_t^{I,B_1}(x,y)=\frac{1}{(2\pi(1-e^{-2t}))^{d/2}}e^{-\frac{|e^{tB_1}x-y|^2}{1-e^{-2t}}},\quad x,y\in\mathbb R^d\mbox{ and }t>0.
\end{align*}

As it was mentioned in the introduction in order to study $L^p$-boundedness properties of the harmonic analysis operators in the Ornstein-Uhlenbeck setting, those operators are decomposed in a local part and a global part.

Let $\delta >0$. We define the sets
\begin{align*}
L_{\color{green} \sigma}=\left\{(x,y)\in\mathbb R^d\times\mathbb R^d:\;|x-y|\leq{\color{green} \sigma}\min\{1,|x+y|^{-1}\}\right\},
\end{align*}
and $G_{\color{green} \sigma}=(\mathbb R^d\times\mathbb R^d)\setminus L_{\color{green} \sigma}$. $L_{\color{green} \sigma}$ and $G_{\color{green} \sigma}$ are named the ${\color{green} \sigma}$-local and ${\color{green} \sigma}$-global region respectively.

The integral kernel $h_t^{I,B_1}$ of $\mathcal H_t^{I,B_1}$, $t>0$, can be estimated in different ways on $L_{\color{green} \sigma}$ and $G_{\color{green} \sigma}$.

\begin{lem}\label{Lem2.1}(\cite[Lemma 3.3]{MN2}).
For every $\sigma >0$, there exists $C>0$ such that
$$
0\leq h_t^{I,B_1}(x,y)\leq \frac{C}{(1-e^{-2t})^{d/2}}e^{-\frac{|x-y|^2}{2(1-e^{-2t})}},\quad (x,y)\in L_\sigma,\;x\neq y\mbox{ and }t>0.
$$
\end{lem}

In order to obtain a more manageable form of the kernel $ h_t^{I,B_1}$ in the symmetric case, that is, when $R=0$, the following change of variable due to S. Meda was introduced in \cite{GCMMST1}
$$
\tau(s)=\log\frac{1+s}{1-s},\quad s\in(0,1).
$$
We observe that $\tau$ maps $(0,1)$ onto $(0,\infty)$. 

For every $s\in(0,1)$ we consider the quadratic form $Q_s$ defined by
$$Q_s(x,y)=|(1+s)x-(1-s)y|^2,\quad x,y\in\mathbb R^d.$$
As in \cite[p. 190]{MN1} if $J$ is an interval in $(0,\infty)$ and $D>0$ we define $J_D^\# =\cup_{n\in\mathbb N}(J+nD)$.

After a careful reading of the proof of \cite[Lemmas 5.5 and 5.6]{MN1} we can see that with minor modifications in those ones the following properties can be proved.

\begin{lem}\label{Lem2.2} Let $\delta\in (0,1)$.
\begin{enumerate}
    \item[(i)] There exists $C$ and $t_0>0$ such that
    \begin{align}\label{2.2}
    h_{\tau(s)}^{I,B_1}(x,y)\leq Cs^{-\frac{d}{2}}e^{|x|^2-|y|^2-
    \frac{\delta}{4s}Q_s(x,y)},\quad x,y\in\mathbb R^d\mbox{ and }s\in\tau^{-1}((0,t_0)).
    \end{align}
    \item[(ii)] Suppose that the one-parameter group of rotations $\{e^{tR}\}_{t\in\mathbb R}$ generated by the matrix $R$ is periodic of period $D$. Then, there exists an interval $J=(a,b)$, with $0<a<b<\infty$, and $C>0$ such that
 \begin{align}\label{2.3}
 h_{\tau(s)}^{I,B_1}(x,y)\leq Cs^{-\frac{d}{2}}e^{|x|^2-|y|^2-\frac{\delta}{4s}Q_s(x,y)},\quad x,y\in \mathbb R^d\mbox{ and }s\in\tau^{-1}(J_D^\#).
\end{align}
\end{enumerate}
Here $C$, $t_0$, $a$ and $b$ depend on $\delta$.
\end{lem}

Note that (\ref{2.2}) and (\ref{2.3}) also hold when $t_0$ is replaced by $t_1\in (0,t_0)$ and $J$ is replaced by an interval $J_1\subset J$, respectively. In the sequel we consider $t_0>0$ and an interval $J=(a,b)$, with $0<a<b<\infty$, satisfying (\ref{2.2}) and (\ref{2.3}), respectively, and such that there exist $n,m\in\mathbb N$ and $\beta >0$ for which
\begin{equation}\label{covering}
(0,\infty)\setminus \mathcal{N}=\Big[\bigcup_{k=0}^n((0,t_0)+kt_0)\Big]\cup\Big[\bigcup_{\ell=0}^m(J_D^\# +\ell\beta)\Big],
\end{equation}
 for certain $\mathcal{N}\subset (0,\infty )$ of measure zero, and being a disjoint union.
Let $\sigma >0$. We choose an smooth function $\varphi$ in $\mathbb R^d\times\mathbb R^d$ satisfying that
\begin{enumerate}
    \item[(i)] $0\leq \varphi\leq 1$, $x,y\in\mathbb R^d$;
    \item[(ii)] $\varphi(x,y)=1$, $(x,y)\in L_\sigma$, and $\varphi(x,y)=0$, $(x,y)\notin L_{2\sigma}$;
    \item[(iii)] $|\nabla_x\varphi(x,y)|+|\nabla_y\varphi(x,y)|\leq\frac{C}{|x-y|}$, $x,y\in\mathbb R^d$, $x\neq y$.
\end{enumerate}

By $L^0(\mathbb R^d)$ we denote the space of Lebesgue measurable functions in $\mathbb R^d$ and we represent by $C_c^\infty(\mathbb R^d)$ the space of smooth functions in $\mathbb R^d$ having compact support. Suppose that $T$ is a linear or sublinear operator from $C_c^\infty(\mathbb R^d)$ into $L^0(\mathbb R^d)$. We define the local part $T_{\rm loc}$ of $T$ by
$$
T_{\rm loc}(f)(x)=T(f(\cdot)\varphi(x,\cdot))(x),\quad x\in\mathbb R^d,
$$
and the global part $T_{\rm glob}$ of $T$ by
$$
T_{\rm glob}(f)(x)=T(f)(x)-T_{\rm loc}(f)(x),\quad x\in\mathbb R^d.
$$
The following results were proved in \cite{GCMMST2} and they will be useful to prove that the global parts of the operators in Theorem \ref{Th1.1} are bounded from $L^1(\mathbb R^d,\gamma_{\infty,1})$ into $L^{1,\infty}(\mathbb R^d,\gamma_{\infty,1})$. If $x,y\in\mathbb R^d\setminus\{0\}$, $\theta(x,y)$ denotes the angle between $x$ and $y$.

\begin{lem}\label{Lem2.3}(\cite[Lemma 4.1]{GCMMST2})
 For every $\delta >0$ there exists $C>0$ such that
$$\sup_{0<s<1}s^{-d/2}e^{ (-\frac{\delta}{s})Q_s(x,y)}\leq C\min\{(1+|x|)^d, (|x|\sin\theta(x,y))^{-d}\},\;\;\;\;(x,y)\in G_1,\;x\neq 0\neq y.$$
\end{lem}

\begin{lem}\label{Lem2.4}(\cite[Lemma 4.4]{GCMMST2}). The operator $T$ defined by 
$$
T(f)(x)=e^{|x|^2}\int_{\mathbb R^d} \min\{(1+|x|)^d, (|x|\sin\theta(x,y))^{-d}\}f(y)e^{-|y|^2}dy,\quad x\in\mathbb R^d,
$$
is bounded from $L^1(\mathbb R^d,\gamma_{\infty,1})$ into $L^{1,\infty}(\mathbb R^d,\gamma_{\infty,1})$.
\end{lem}

In the study of the local parts of the operators in Theorem \ref{Th1.1} we will use the $L^p$-boundedness properties of the operator $S_\sigma$ defined by
$$
S_\sigma(f)(x)=\int_{\{y\in\mathbb R^d:(x,y)\in L_\sigma\}}\frac{1+|x|}{|x-y|^{d-1}}f(y)dy,\quad x\in\mathbb R^d,
$$
where $\delta >0$. Operators of this type appear also when the symmetric case is considered  (see, for instance, \cite{HMMT}).

\begin{lem}\label{Lem2.5}
Let $\sigma>0$. The operator $S_\sigma$ is bounded from $L^p(\mathbb R^d,dx)$ into itself and also  from $L^p(\mathbb R^d,\gamma_{\infty,1})$ into itself, for every $1\leq p\leq\infty$.
\end{lem}
\begin{proof}
We include a sketch of the proof of this property for the sake of completeness.

We have that $|x+y|=|2x+y-x|\geq 2|x|-|x-y|\geq 2|x|-\sigma\geq |x|$, provided that $|x-y|\leq\sigma \leq |x|$. It follows that
$$
\int_{\{y\in\mathbb R^d:(x,y)\in L_\sigma\}}\frac{1+|x|}{|x-y|^{d-1}}dy\leq C(1+|x|)\times \left\{\begin{array}{ll} \displaystyle \int_0^{\sigma/|x|}dr, & |x|\geq\delta \\
 & \\
 \displaystyle \int_0^\sigma dr, & |x|\leq\delta
  \end{array}\right.\leq C.
$$
Hence, $\displaystyle\sup_{x\in\mathbb R^d}\int_{\{y\in\mathbb R^d:(x,y)\in L_\sigma\}}\frac{1+|x|}{|x-y|^{d-1}}dy<\infty$ and in a similar way  
$
\displaystyle\sup_{y\in\mathbb R^d}\int_{\{x\in\mathbb R^d:(x,y)\in L_\sigma\}}\frac{1+|x|}{|x-y|^{d-1}}dx<\infty.
$

By using interpolation we deduce that the operator $S_\sigma$ is bounded from $L^p(\mathbb R^d,dx)$ into itself, for every $1\leq p\leq\infty$. According to \cite[Lemma 3.6]{GCMST2} $S_\sigma$ is bounded from $L^p(\mathbb R^d,\gamma_{\infty,1})$ into itself, for every $1\leq p\leq\infty$.
\end{proof}

We are going to explain the method we use to prove Theorem \ref{Th1.1}. We extend the procedure developed by Mauceri and Noselli (\cite{MN2} and \cite{MN1}).

Suppose that $X$ is a Banach space of complex functions defined in $(0,\infty)$. Let $k\in\mathbb N$ and $j=1,...,d$. We consider the operator $T_{k,j}^X$ defined by
$$
T_{k,j}^X(f)(x)=\big\|t\rightarrow t^{k+1}\partial_t^k\partial_{x_j}P_t^{I,B_1}(f)(x)\big\|_X,\quad x\in\mathbb R^d.
$$
It is clear that $T_{k,j}^X$ reduces to $P_{*,k,j}^{I,B_1}$ and $g_{k,j}^{I,B_1}$ when $X=L^\infty((0,\infty),dt)$ and $X=L^2((0,\infty),\frac{dt}{t})$, respectively. Furthermore, let $\rho >2$. We consider on the space $C(0,\infty)$ of continuous functions on $(0,\infty)$ and the seminorm $V_\rho$ defined in (\ref{1.2}). By identifying those functions in $C(0,\infty)$ that differ in a constant the space $V_\rho(0,\infty)$ consisting of all those $g\in C(0,\infty)$ such that $V_\rho(g)<\infty$ endowed with $V_\rho$ is a Banach space. We have that
$T_{k,j}^{V_\rho(0,\infty)}=V_{\rho,k,j}^{I,B_1}$. 

Let $f\in C_c^\infty(\mathbb{R}^d)$. We can write
\begin{align}\label{2.5}
   \partial_t^k\partial_{x_j}P_t^{I,B_1}(f)(x)=\int_{\mathbb R^d}f(y)  \partial_t^k\partial_{x_j}P_t^{I,B_1}(x,y)dy,\quad x\in\mathbb R^d\mbox{ and }t>0,
\end{align}
where $P_t^{I,B_1}(x,y)$, $x,y\in\mathbb R^d$ and $t>0$, denotes the Poisson integral kernel and 
\begin{align*}
   \partial_t^k\partial_{x_j}P_t^{I,B_1}(x,y)=\frac{1}{2\sqrt{\pi}}\int_0^\infty\partial_t^k[te^{-t^2/4u}]\partial_{x_j}h_u^{I,B_1}(x,y)\frac{du}{u^{\frac{3}{2}}},\quad x,y\in\mathbb R^d\mbox{ and }t>0.
\end{align*}
Differentiations under the integral sign are justified. Indeed, we have 
\begin{align*}
   |\partial_{x_j}h_u^{I,B_1}(x,y)|\leq C\frac{|e^{uB_1^*}||e^{uB_1}x-y|}{(1-e^{-2u})^{d/2+1}}e^{-|e^{uB_:1}x-y|^2/(1-e^{-2u})},\quad x,y\in\mathbb R^d\mbox{ and }u>0.
\end{align*}
Since $|e^{uB^*}|\leq e^{-u}$, $u>0$, we get
$$
|\partial _{x_j}h_u^{I,B_1}(x,y)|\leq C\frac{e^{-u}}{(1-e^{-2u})^{(d+1)/2}},\quad x,y\in \mathbb{R}^d\mbox{ and }t>0.
$$

By using \cite[Lemma 4]{BCCFR} we obtain, for each $x,y\in \mathbb{R}^d$ and $t>0$,
\begin{align*}
    \int_0^\infty |\partial_t^k[te^{-\frac{t^2}{4u}}]||\partial_{x_j}h_u^{I,B_1}(x,y)|\frac{du}{u^{\frac{3}{2}}}&\leq C\int_0^\infty \frac{e^{-\frac{t^2}{8u}-u}}{u^{(k+2)/2}(1-e^{-2u})^{(d+1)/2}}du\\
    &\hspace{-1cm}\leq C\left(\int_1^\infty e^{-u}du+\int_0^1\frac{e^{-\frac{t^2}{8u}}}{u^{(k+d+3)/2}}du\right)\leq C(1+t^{-k-d-1}).
\end{align*}
Since $f\in C_c^\infty (\mathbb{R}^d)$ \eqref{2.5} holds.

Suppose that $E\subset(0,\infty)$ and $h\geq 0$. We define the operator
$$
S_{k,j}^{E,h}(f)(x,t)=\int_{\mathbb{R}^d}\mathfrak{s}_{k,j}^{E,h}(x,y,t)f(y)dy,\quad x\in \mathbb{R}^d\mbox{ and }t>0.
$$
where
$$
\mathfrak{s}_{k,j}^{E,h}(x,y,t)=\frac{t^{k+1}}{2\sqrt{\pi}}\int_E\partial _t^k[te^{-\frac{t^2}{u+h}}]\partial_{x_j}h_u^{I,B_1}(x,y)\frac{du}{(u+h)^{\frac{3}{2}}},\quad x,y\in \mathbb{R}^d \mbox{ and }t>0.
$$

Assume that $X$ is one of the following Banach spaces: $L^\infty ((0,\infty ),dt)$, $L^2((0,\infty ),dt/t)$ and $V_\rho (0,\infty )$. We consider the operator
$$
\widetilde{S}_{k,j}^{E,h}(f)(x)=\left\|S_{k,j}^{E,h}(f)(x,\cdot )\right\|_X,\quad x\in \mathbb{R}^d.
$$
\noindent {\sc Claim 1} {\it Let $1\leq p<\infty$. We define $\delta_p$ as follows
$$
\delta_p=\left\{\begin{array}{ll} 
\displaystyle \frac{1}{10}, & p=1, \\[0.3cm]
\displaystyle \frac{1}{2}\Big(1+\max\big\{1-\frac{1}{d},1-\frac{1}{p}\big\}\Big), & 1<p<\infty.
\end{array}\right.
$$
We denote $E_p$ the sets $(0,t_0)$ or $J_D^\#$ in Lemma \ref{Lem2.2} associated to $t_0$ and satisfying the covering property \eqref{covering}.

Then, for every $h\geq 0$, the operator $\widetilde{S}_{k,j}^{E_p,h}$ is bounded from $L^p(\mathbb{R}^d,\gamma_{\infty,1})$ into itself, when $1<p<\infty$, and from $L^1(\mathbb{R}^d,\gamma_{\infty ,1})$ into $L^{1,\infty}(\mathbb{R}^d,\gamma_{\infty,1})$, when $p=1$.}

Suppose that the claim has been proved. Since 
$$
(0,\infty)\setminus{\mathcal{N}}=\Big[\bigcup_{\ell =0}^n((0,t_0)+\ell t_0)\Big]\cup \Big[\bigcup_{\ell =0}^m(J_P^\#+\ell \beta)\Big]
$$
we can write
\begin{align*}
    \int_0^\infty \partial_t^k[te^{-\frac{t^2}{4u}}]\partial_{x_j}h_u^{I,B_1}(x,y)\frac{du}{u^{\frac{3}{2}}}&=\left(\sum_{\ell =0}^n\int_{(0,t_0)+\ell t_0}+\sum_{\ell =0}^m\int_{J_D^\#+\ell \beta}\right)\partial_t^k[te^{-\frac{t^2}{4u}}]\partial_{x_j}h_u^{I,B_1}(x,y)\frac{du}{u^{\frac{3}{2}}}\\
    &=\sum_{\ell =0}^n\int_{(0,t_0)}\partial_t^k[te^{-\frac{t^2}{4(u+\ell t_0)}}]\partial _{x_j}h_{u+\ell t_0}^{I,B_1}(x,y)\frac{du}{(u+\ell t_0)^{\frac{3}{2}}}\\
    &\quad +\sum_{\ell =0}^m\int_{J_D^\#}\partial_t^k[te^{-\frac{t^2}{2(u+\ell \beta)}}]\partial _{x_j}h_{u+\ell \beta}^{I,B_1}(x,y)\frac{du}{(u+\ell \beta)^{\frac{3}{2}}},\quad x,y\in \mathbb{R}^d\mbox{ and }t>0.
\end{align*}
By using the semigroup property of $\{H_t^{I,B_1}\}_{t>0}$ we deduce that
$$
t^{k+1}\partial_t^k\partial _{x_j}P_t^{I,B_1}(f)(x)=\sum_{\ell =0}^nS_{k,j}^{(0,t_0),\ell t_0}(H_{\ell t_0}^{I,B_1}(f))(x,t)+\sum_{\ell =0}^mS_{k,j}^{J_D^\#,\ell \beta}(H_{\ell \beta}^{I,B_1}(f))(x,t),\quad x\in \mathbb{R}^d\mbox{ and }t>0.
$$
Since the semigroup $\{H_t^{I,B_1}\}_{t>0}$ is contractive in $L^p(\mathbb{R}^d,\gamma_{\infty ,1})$, for every $1\leq p<\infty $, the Claim 1 allows us to conclude that the operator $T_{k,j}^X$ is bounded from $L^p(\mathbb{R}^d,\gamma_{\infty ,1})$ into itself, for every $1<p<\infty$, and from $L^1(\mathbb{R}^d,\gamma_{\infty ,1})$ into $L^{1,\infty }(\mathbb{R}^d,\gamma_{\infty ,1})$.

Our objective is to prove the Claim 1. In order to see the $L^p$-boundedness properties of the operator $\widetilde{S}_{k,j}^{E,h}$ we study separately the local part and the global part of $\widetilde{S}_{k,j}^{E,h}$. We analyze firstly the local part $\widetilde{S}_{k,j,{\rm loc}}^{E,h}$ of $\widetilde{S}_{k,j}^{E,h}$. We consider the operator
$$
U_{k,j}^{E,h}(f)(x,t)=\int_{\mathbb{R}^d}\mathfrak{U}_{k,j}^{E,h}(x-y,t)f(y)dy, \quad x\in \mathbb{R}^d\mbox{ and }t>0,
$$
where 
$$
\mathfrak{U}_{k,j}^{E,h}(z,t)=\frac{t^{k+1}}{2\sqrt{\pi }}\int_E\partial _t^k\big[te^{-\frac{t^2}{4(u+h)}}\big]\partial _{x_j}W_u(z)\frac{du}{(u+h)^{\frac{3}{2}}},\quad z\in \mathbb{R}^d\mbox{ and }t>0.
$$
Here $W_u$, $u>0$ denotes the classical heat kernel given by
$$
W_u(z)=\frac{e^{-\frac{|z|^2}{2u}}}{(2\pi u)^{d/2}},\quad z\in \mathbb{R}^d\mbox{ and }u>0.
$$

We also define 
$$
\widetilde{U}_{k,j}^{E,h}(f)(x)=\left\|U_{k,j}^{E,h}(f)(x,\cdot )\right\|_X,\quad x\in \mathbb{R}^d.
$$
By using Minkowski inequality we deduce that
$$
\Big|\widetilde{S}_{k,j,{\rm loc}}^{E,h}(f)(x)-\widetilde{U}_{k,j,{\rm loc}}^{E,h}(f)(x)\Big|\leq \int_{\mathbb{R}^d}\Big\|\mathfrak{s}_{k,j}^{E,h}(x,y,\cdot )-\mathfrak{U}_{k,j}^{E,h}(x-y,\cdot )\Big\|_X\varphi (x,y)|f(y)|dy,\quad x\in \mathbb{R}^d.
$$

\noindent{\sc Claim 2}. {\it Let $E\subset(0,\infty)$ and $h\geq 0$. The operator $D_{k,j,{\rm loc}}^{E,h}$ defined by 
$$
D_{k,j,{\rm loc}}^{E,h}(f)(x)=\int_{\mathbb{R}^d}\big\|\mathfrak{s}_{k,j}^{E,h}(x,y,\cdot )-\mathfrak{U}_{k,j}^{E,h}(x-y,\cdot )\Big\|_X\varphi (x,y)f(y)dy,\quad x\in \mathbb{R}^d,
$$
is bounded from $L^p(\mathbb{R}^d,dx)$ into itself, and from $L^p(\mathbb{R}^d,\gamma_{\infty,1})$ into itself, for every  $1\leq p\leq \infty$}.

\noindent{\sc Claim 3}. {\it Let $h\geq 0$ and $E\subset (0,\infty)$ such that $E\subset (0,\eta,)$ or $E\subset (\eta,\infty)$, for some $\eta >0$.  The operator $\widetilde{U}_{k,j,{\rm loc}}^{E,h}$ is bounded from $L^p(\mathbb{R}^d,dx)$ into itself, for every $1<p<\infty$, and from $L^1(\mathbb{R}^d,dx)$ into $L^{1,\infty }(\mathbb{R}^d,dx)$.}

By using \cite[Lemma 3.6]{GCMST2} from Claim 3 we deduce that the operator $\widetilde{U}_{k,j,{\rm loc}}^{E,h}$ is bounded from $L^p(\mathbb{R}^d,\gamma_{\infty,1})$ into itself, for every $1<p<\infty$, and from $L^1(\mathbb{R}^d,\gamma_{\infty ,1})$ into $L^{1,\infty }(\mathbb{R}^d,\gamma_{\infty ,1})$. Then, according to Claim 2 it follows that the operator $\widetilde{S}_{k,j,{\rm loc}}^{E,h}(f)$ is bounded from $L^p(\mathbb{R}^d,\gamma_{\infty,1})$ into itself, for every $1<p<\infty$, and from $L^1(\mathbb{R}^d,\gamma_{\infty ,1})$ into $L^{1,\infty }(\mathbb{R}^d,\gamma_{\infty ,1})$.

Minkowski's inequality leads to 
$$
\widetilde{S}_{k,j,{\rm glob}}^{E,h}(f)(x)\leq \int_{\mathbb{R}^d}\big\|\mathfrak{s}_{k,j}^{E,h}(x,y,\cdot )\big\|_X(1-\varphi (x,y))|f(y)|dy,\quad x\in \mathbb{R}^d.
$$

\noindent {\sc Claim 4}. {\it Let $h\geq 0$. Assume that $1\leq p <\infty$ and that $E_p$ is the set associated to $\delta_p$ as in Claim 1. Then, the operator $\widetilde{V}_{k,j,{\rm glob}}^{E_p,h}$ defined by
$$
\widetilde{V}_{k,j,{\rm glob}}^{E_p,h}(f)(x)=\int_{\mathbb{R}^d}\big\|\mathfrak{s}_{k,j}^{E_p,h}(x,y,\cdot )\big\|_X(1-\varphi (x,y))|f(y)|dy,\quad x\in \mathbb{R}^d,
$$
is bounded from $L^p(\mathbb{R}^d,\gamma_{\infty,1})$ into itself, when $1<p<\infty$, and from $L^1(\mathbb{R}^d,\gamma_{\infty ,1})$ into $L^{1,\infty }(\mathbb{R}^d,\gamma_{\infty ,1})$.}

Thus, Claim 1 is proved when we establish Claims 2, 3 and 4.

Let $k\in \mathbb{N}$, $j=1,...,d$, and $M\geq 1$. We consider the operator $\mathbb{T}_{k,j,M}^X$ defined by $$
\mathbb{T}_{k,j,M}^X(f)(x)=\left\|t\rightarrow t^{k+\frac{1}{2}}\partial_t^k\partial _{x_j}(I+t\mathcal{L}^{I,B_1})^{-M}(f)(x)\right\|_X,\quad x\in \mathbb{R}^d.
$$
$\mathbb{T}_{k,j,M}^X$ reduces to $S_{*,k,j,M}^{I,B_1}$, $G_{k,j,M}^{I,B_1}$ and $\mathcal{V} _{\rho ,k,j,M}^{I,B_1}$ when $X=L^\infty ((0,\infty ),dt)$, $X=L^2((0,\infty ),\frac{dt}{t})$ and $X=V_\rho (0,\infty )$, respectively.

We have that
$$
(I+t\mathcal{L}^{I,B_1})^{-M}f=\frac{1}{\Gamma (M)}\int_0^\infty e^{-u}\mathcal{H}_{ut}^{I,B_1}(f)u^{M-1}du.
$$
It is clear that 
$$
\sup_{t>0}|(I+t\mathcal{L}^{I,B_1})^{-M}f|\leq \sup_{t>0}|\mathcal{H}_t^{I,B_1}(f)|.
$$

According to \cite[Corollary 4.3 and Theorem 5.1]{MN1} we have that the maximal operator $\sup_{t>0}|(I+t\mathcal{L}^{I,B_1})^{-M}|$ is bounded from $L^p(\mathbb{R}^d,\gamma_{\infty ,1})$ into itself, for every $1<p<\infty$, provided that the Ornstein-Uhlenbeck operator $\mathcal{L}^{I,B_1}$ is normal, and from $L^1(\mathbb{R}^d,\gamma_{\infty ,1})$ into $L^{1,\infty }(\mathbb{R}^d,\gamma_{\infty ,1})$ when the matrix $R$ generates a periodic group $\{e^{Rt}\}_{t>0}$. In order to establish the $L^p$-boundedness properties for the maximal operators $S_{*,k,j,M}^{I,B_1}$ we need to work harder because, as far as we know, the $L^p$-boundedness properties for the corresponding maximal operators involving the heat semigroup $\{\mathcal{H}_t^{I,B_1}\}_{t>0}$ have not been studied. 

Let $f\in C_c^\infty (\mathbb{R}^d)$. We can write 
\begin{align*}
    (I+t\mathcal{L}^{I,B_1})^{-M}(f)(x)&=\frac{t^{-M}}{\Gamma (M)}\int_0^\infty e^{-\frac{u}{t}}\mathcal{H}_u^{I,B_1}(f)(x)u^{M-1}du\\
    &=\frac{t^{-M}}{\Gamma (M)}\int_{\mathbb{R}^d}f(y)\int_0^\infty e^{-\frac{u}{t}}h_u^{I,B_1}(x,y)u^{M-1}dudy,\quad x\in \mathbb{R}^d\mbox{ and }t>0.
\end{align*}
Then, for every $x\in \mathbb{R}^d$ and $t>0$,
$$
t^{k+\frac{1}{2}}\partial _t^k\partial _{x_j}(I+t\mathcal{L}^{I,B_1})^{-M}(f)(x)=\frac{t^{k+\frac{1}{2}}}{\Gamma (M)}\int_{\mathbb{R}^d}f(y)\int_0^\infty \partial_t^k[t^{-M}e^{-\frac{u}{t}}]\partial _{x_j}h_u^{I,B_1}(x,y)u^{M-1}dudy.
$$
Differentiation under the integral can  be justified as above by considering that $M>(d+1)/2$.

We take $E\subset(0,\infty)$ and $h\geq 0$.  As in the previous case, we define the operator $\mathbb{S}_{k,j,M}^{E,h}$ by
$$
\mathbb{S}_{k,j,M}^{E,h}(f)(x,t)=\int_{\mathbb{R}^d}\mathfrak{s}_{k,j,M}^{E,h}(x,y,t)f(y)dy,\quad x\in \mathbb{R}^d\mbox{ and }t>0,
$$
where
$$
\mathfrak{s}_{k,j,M}^{E,h}(x,y,t)=\frac{t^{k+\frac{1}{2}}}{\Gamma (M)}\int_0^\infty \partial _t^k[t^{-M}e^{-(u+h)/t}]\partial _{x_j}h_u^{I,B_1}(x,y)(u+h)^{M-1}du,\quad x,y\in \mathbb{R}^d\mbox{ and }t>0.
$$
In order to prove the $L^p$-boundedness properties of the operator $\mathbb{T}_{k,j,M}^X$ where $X=L^\infty ((0,\infty ),dt)$, $X=L^2((0,\infty ),\frac{dt}{t})$ and $X=V_\rho (0,\infty )$, we can proceed by following the same steps than in the previous case  by considering the operator
$$
\widetilde{\mathbb{S}}_{k,j,M}^{E,h}(f)(x)=\big\|\mathbb{S}_{k,j,M}^{E,h}(f)(x,\cdot)\big\|_X, \quad x\in \mathbb{R}^d.
$$

\begin{rem}
Let $k\in \mathbb{N}$, $\alpha =(\alpha _1,...,\alpha _d)\in \mathbb{N}^d$ such that $k+\widehat{\alpha}>0$ when $\widehat{\alpha}=\alpha _1+...+\alpha _d$. Assume that $X=L^\infty ((0,\infty ),dt)$, $X=L^2((0,\infty ),\frac{dt}{t})$ or $X=V_\rho (0,\infty )$. We define the operator $T_{k,\alpha}^X$ by
$$
T_{k,\alpha }^X(f)(x)=\left\|t\rightarrow t^{k+\widehat{\alpha}}\partial _t^k\partial_x^\alpha P_t^{I,B_1}(f)(x)\right\|_X,\quad x\in \mathbb{R}^d.
$$
Here $\partial _x^\alpha=\frac{\partial ^{\widehat{\alpha}}}{\partial x_1^{\alpha_1}\cdots \partial x_d^{\alpha _d}}$.

It is natural to ask if $L^p$-boundedness properties of this operator when $\widehat{\alpha}\not =1$ can be proved by using the procedure in this paper. At this moment we can not apply our procedure because we do not know how to deal with the global parts of the operators when $\widehat{\alpha }\not =1$.

We now comment about some special cases. We consider $\widehat{\alpha}=0$ and $X=L^\infty ((0,\infty ),dt)$. By using the method in \cite[\S 4]{LiSj} we can see that 
$$
T_{k,0}^{L^\infty ((0,\infty ),dt)}(f)\leq C\sup_{t>0}\frac{1}{t}\left|\int_0^t\mathcal{H}_s^{I,B_1}(f)ds\right|.
$$
Since $\{\mathcal{H}_t^{I,B_1}\}_{t>0}$ is contractive in $L^p(\mathbb{R}^d,\gamma_{\infty ,1})$, $1\leq p\leq \infty$, the Hopf-Dunford-Schwartz ergodic theorem (\cite[Lemma VIII.7.6 and Theorem VIII.7.7]{DS}) leads to $T_{k,0}^{L^\infty ((0,\infty ),dt)}$ is bounded from $L^p(\mathbb{R}^d,\gamma_{\infty ,1})$ into itself, for every $1<p<\infty$, and from $L^1(\mathbb{R}^d,\gamma_{\infty ,1})$ into $L^{1,\infty }(\mathbb{R}^d,\gamma_{\infty ,1})$.

We can write
$$
\partial _t^kP_t^{I,B_1}(x,y)=\frac{1}{\sqrt{\pi}}\int_0^\infty \partial _t^{k-1}[t e^{-\frac{t^2}{4u}}]\partial _uh_u^{I,B_1}(x,y)\frac{du}{\sqrt{u}},\quad x,y\in \mathbb{R}^d\mbox{ and }t>0.
$$

We have that
$$
\partial _uh_u^{I,B_1}(x,y)=\frac{1}{2}\sum_{j=1}^d\partial _{x_j}^2h_u^{I,B_1}(x,y)+\langle B_1x, \nabla_x h_u^{I,B_1}(x,y)\rangle,\quad x,y \in \mathbb{R}^d\mbox{ and }u>0.
$$
Then, the cases $\widehat {\alpha }=0$ and $\widehat{\alpha}=2$ are connected.
\end{rem}

The arguments used in the symmetric Ornstein-Uhlenbeck setting in \cite{HMMT}, \cite{P2}, \cite{P1} and \cite{PS} do not work for the global parts of the operator $T_{k,\alpha}^X$ in the nonsymmetric context. Our objective in a next paper is to establish $L^p$-boundedness properties of $T_{k,\alpha}^X$-type operators for general nonsymmetric Ornstein-Uhlenbeck operators by using some of the ideas developed by Casarino, Ciatti and Sj\"ogren (\cite{CCS2}, \cite{CCS1} and \cite{CCS3}).

\section{Proof of Claims 2 and 4}
Our objective in this section is to prove Claims 2 and 4 stated in the previous section for the operators in Theorem \ref{Th1.1}.
\subsection{Proof of Claim 2} We consider firstly the operators $P_{*,k,j}^{I,B_1}$, $g_{k,j}^{I,B_1}$ and $V_{\rho ,k,j}^{I,B_1}$. In the sequel $X$ represents one of the following Banach spaces: $L^\infty ((0,\infty ),dt)$, $L^2((0,\infty ),\frac{dt}{t})$ and $V_\rho (0,\infty )$. We are going to study the operator
$$
D_{k,j,{\rm loc}}^{E,h}(f)(x)=\int_{\mathbb{R}^d}\Big\|\mathfrak{s}_{k,j}^{E,h}(x,y,\cdot)-\mathfrak{U}_{k,j}^{E,h}(x-y,\cdot )\Big\|_X\varphi (x,y)f(y)dy,\quad x\in \mathbb{R}^d,
$$
where $E\subset(0,\infty)$ and $h\geq 0$. The definitions can be found in Section 2.

We can write, for each $x,y\in \mathbb{R}^d$ and $t>0$,
$$
\mathfrak{s}_{k,j}^{E,h}(x,y,t)-\mathfrak{U}_{k,j}^{E,h}(x-y,t)=\frac{t^{k+1}}{2\sqrt{\pi}}\int_E\frac{\partial _t^k[te^{-\frac{t^2}{4(u+h)}}]}{(u+h)^{\frac{3}{2}}}\big(\partial _{x_j}h_u^{I,B_1}(x,y)-\partial _{x_j}W_u(x-y)\big)du.
$$
By using Minkowski inequality we get
\begin{equation}\label{3.1}
\big\|\mathfrak{s}_{k,j}^{E,h}(x,y,\cdot)-\mathfrak{U}_{k,j}^{E,h}(x-y,\cdot)\big\|_X\leq C\int_E \frac{\|t^{k+1}\partial _t^k[te^{-\frac{t^2}{4(u+h)}}]\|_X}{(u+h)^{\frac{3}{2}}}\big|\partial _{x_j}h_u^{I,B_1}(x,y)-\partial _{x_j}W_u(x-y)\big|du,\quad x,y\in\mathbb{R}^d.
\end{equation}
According to \cite[Lemma 4]{BCCFR} it follows that
\begin{equation}\label{3.2}
\big|\partial _t^k[te^{-\frac{t^2}{4(u+h)}}]\big|\leq C\frac{e^{-\frac{t^2}{8(u+h)}}}{(u+h)^{(k-1)/2}},\quad t,u\in (0,\infty ).
\end{equation}
We have that
\begin{equation}\label{3.2'}
\big|t^{k+1}\partial _t^k[te^{-\frac{t^2}{4(u+h)}}]\big|\leq C\frac{t^{k+1}e^{-\frac{t^2}{8(u+h)}}}{(u+h)^{(k-1)/2}}\leq C(u+h),\quad t,u\in (0,\infty ).
\end{equation}
By using again \eqref{3.2} we obtain
\begin{equation}\label{3.2.1}
   \left\|t^{k+1}\partial _t^k[te^{-\frac{t^2}{4(u+h)}}]\right\|_{L^2((0,\infty ),\frac{dt}{t})}\leq \frac{C}{(u+h)^{(k-1)/2}}\left(\int_0^\infty |t^{k+1}e^{-\frac{t^2}{8(u+h)}}|^2\frac{dt}{t}\right)^{\frac{1}{2}}\leq C(u+h),\quad u>0.
\end{equation}

Suppose that $g:(0,\infty)\longrightarrow \mathbb{C}$ is a derivable function. If $t_1>t_2>\cdots >t_k>0$ we have that
$$
\left(\sum_{j=1}^{k-1}|g(t_{j+1})-g(t_j)|^\rho \right)^{1/\rho }=\left(\sum_{j=1}^{k-1}\Big|\int_{t_j}^{t_{j+1}}g'(t)dt\Big|^\rho \right)^{1/\rho}\leq \sum_{j=1}^{k-1}\int_{t_j}^{t_{j+1}}|g'(t)|dt\leq \int_0^\infty |g'(t)|dt.
$$
Then,
$$
V_\rho (g)\leq \int_0^\infty |g'(t)|dt.
$$
From \eqref{3.2} we deduce that
\begin{align}\label{3.2.2}
   \left\|t^{k+1}\partial _t^k[te^{-\frac{t^2}{4(u+h)}}]\right\|_{V_\rho}&\leq \int_0^\infty \Big|\partial _t(t^{k+1}\partial _t^k(te^{-\frac{t^2}{4(u+h)}}))\Big|dt \leq C\left(\int_0^\infty |t^k\partial _t^k(te^{-\frac{t^2}{4(u+h)}})|dt\right.\nonumber\\
   &\left.\quad + \int_0^\infty |t^{k+1}\partial _t^{k+1}(te^{-\frac{t^2}{4(u+h)}})|dt\right)\nonumber\\
   &\leq C\left((u+h)^{-(k-1)/2}\int_0^\infty t^ke^{-\frac{t^2}{8(u+h)}}dt+(u+h)^{-k/2}\int_0^\infty t^{k+1}e^{-\frac{t^2}{4(u+h)}}dt\right)\nonumber\\
   &\leq C(u+h),\quad u>0.
\end{align}
By \eqref{3.1} we get
\begin{align}\label{3.3}
\Big\|\mathfrak{s}_{k,j}^{E,h}(x,y,\cdot )-\mathfrak{U}_{k,j}^{E,h}(x-y,\cdot )\Big\|_X\leq C\int_0^\infty \frac{1}{\sqrt{u}}|\partial _{x_j}h_u^{I,B_1}(x,y)-\partial _{x_j}W_u(x-y)|du,\quad x,y\in \mathbb{R}^d.
\end{align}
We are going to see that
\begin{equation}\label{3.4}
\int_0^\infty \frac{1}{\sqrt{u}}|\partial_{x_j}h_u^{I,B_1}(x,y)-\partial _{x_j}W_u(x-y)|du\leq C\frac{1+|x|}{|x-y|^{d-1}},\quad (x,y)\in L_2.
\end{equation}
When $B_1=-I$ the operator $\mathcal{L}^{I,B_1}$ is the symmetric Ornstein-Uhlenbeck and we have that
$$
h_u^{I,-I}(x,y)=\frac{e^{-\frac{|e^{-u}x-y|^2}{1-e^{-2u}}}}{(2\pi (1-e^{-2u}))^{d/2}},\quad x,y\in \mathbb{R}^d\mbox{ and }u>0.
$$
By using \cite[Lemma 3.4]{HTV} we obtain
$$
\int_0^\infty \frac{1}{\sqrt{u}}|\partial_{x_j}h_u^{I,-I}(x,y)-\partial _{x_j}W_u(x-y)|du\leq C\frac{1+|x|}{|x-y|^{d-1}},\quad (x,y)\in L_2. 
$$
Then, \eqref{3.4} will be proved when we see that
\begin{equation}\label{3.5}
\int_0^\infty \frac{1}{\sqrt{u}}|\partial_{x_j}h_u^{I,B_1}(x,y)-\partial _{x_j}h_u^{I,-I}(x,y)|du\leq C\frac{1+|x|}{|x-y|^{d-1}},\quad (x,y)\in L_2.
\end{equation}
Since $R+R^*=0$ we get
$$
|e^{uB_1}x-y|^2=\langle e^{-uR}(e^{-u}x-e^{uR}y),e^{-uR}(e^{-u}x-e^{uR}y)\rangle=|e^{-u}x-e^{uR}y|^2,\quad x,y\in \mathbb{R}^d\mbox{ and }u>0.
$$
We have that
\begin{align}\label{3.5.1}
\partial_{x_j}h_u^{I,B_1}(x,y)=-\frac{2e^{-u}}{(2\pi )^{\frac{d}{2}}}(e^{-u}x_j-(e^{uR}y)_j)\frac{e^{-\frac{|e^{-u}x-e^{uR}y|^2}{1-e^{-2u}}}}{(1-e^{-2u})^{\frac{d}{2}+1}},\quad x,y\in \mathbb{R}^d\mbox{ and }u>0,
\end{align}
and 
\begin{equation}\label{3.5.2}
\partial_{x_j}h_u^{I,-I}(x,y)=-\frac{2e^{-u}}{(2\pi )^{\frac{d}{2}}}(e^{-u}x_j-y_j)\frac{e^{-\frac{|e^{-u}x-y|^2}{1-e^{-2u}}}}{(1-e^{-2u})^{\frac{d}{2}+1}},\quad x,y\in \mathbb{R}^d\mbox{ and }u>0.
\end{equation}
We can write, for every $x,y\in \mathbb{R}^d$ and $u>0$,
\begin{align*}
    \partial_{x_j}h_u^{I,B_1}(x,y)-\partial _{x_j}h_u^{I,-I}(x,y)&=-\frac{2e^{-u}}{(2\pi )^{\frac{d}{2}}}\frac{1}{(1-e^{-2u})^{\frac{d}{2}+1}}\\
    &\times \left[\Big(e^{-\frac{|e^{-u}x-e^{uR}y|^2}{1-e^{-2u}}}-e^{-\frac{|e^{-u}x-y|^2}{1-e^{-2u}}}\Big)(e^{-u}x_j-(e^{uR}y)_j)+e^{-\frac{|e^{-u}x-y|^2}{1-e^{-2u}}}(y_j-(e^{uR}y)_j)\right].
\end{align*}
We need to establish some estimations. We have that
$$
|e^{-a}-e^{-b}|\leq e^{-\min\{a,b\}}|a-b|,\quad a,b>0.
$$
Then, according to \cite[Lemma 3.3 (i)]{MN2} we obtain
$$
\left|e^{-\frac{|e^{-u}x-e^{uR}y|^2}{1-e^{-2u}}}-e^{-\frac{|e^{-u}x-y|^2}{1-e^{-2u}}}\right|\leq C
\frac{e^{-c\frac{|x-y|^2}{1-e^{-2u}}}}{1-e^{-2u}}
\big||e^{-u}x-y|^2-|e^{-u}x-e^{uR}y|^2\big|,\quad (x,y)\in L_2\mbox{ and }u>0.
$$

We manipulate to get, for every $x,y\in \mathbb{R}^d$ and $u>0$,
\begin{align*}
    |e^{-u}x-y|^2-|e^{-u}x-e^{uR}y|^2&=|e^{-u}x-y|^2 -|e^{-u}x-y-(e^{uR}-I)y|^2\\
    &=-|(e^{uR}-I)y|^2+2\langle e^{-u}x-y,(e^{uR}-I)y\rangle\\
    &=-|(e^{uR}-I)y|^2+2\Big[(e^{-u}-1)\langle x, (e^{uR}-I)y\rangle+\langle x-y,(e^{uR}-I)y\rangle\Big].
\end{align*}
 
It follows that
\begin{equation*}
    \left| |e^{-u}x - y|^2 - |e^{-u}x - e^{uR}y|^2
    \right| \leq C 
    (u^2|y|^2 + u^2|x||y| + u|x-y||y|),
    \quad x,\; y\in\mathbb{R}^d \text{ and }u>0.
\end{equation*}

We conclude that, for $x$, $y\in\mathbb{R}^d$ and $u\in (0,1)$,
\begin{equation*}
    \left|e^{-\frac{|e^{-u}x - e^{uR}y|^2}{1-e^{-2u}}} - 
    e^{-\frac{|e^{-u}x - y|^2}{1-e^{-2u}}}
    \right| \leq  C
    e^{-c \frac{|x-y|^2}{u}} 
    \big[u|y|(|x| + |y|) + |x-y||y|\big].
\end{equation*}

On the other hand we have that, for $x$, $y\in\mathbb{R}^d$ and $u>0$,
$$
|y_j - (e^{uR}y)_j|\leq |(e^{uR}-I)y|\leq Cu|y|,
$$
and 
$$
|e^{-u}x_j - (e^{-uR}y)_j|  \leq|e^{-u}x_j - x_j| + | x_j - y_j | +
|y_j - (e^{uR}y)_j|
\leq C\big[ u(|x| + |y|) + |x-y|\big].
$$

We get, for $x$, $y\in\mathbb{R}^d$ and $ u\in(0,1)$,   
\begin{align*}
        |\partial_{x_j} h^{I, B_1}_u (x,y) - 
        \partial_{x_j} h^{I, -I}_u (x,y)|
        &\leq C \frac{e^{-c\frac{|x-y|^2}{u}}}{u^{\frac{d}{2}+1}}
        \left(
        \big[u|y|(|x|+ |y|) + |x-y||y|\big]
        \big[u(|x| + |y|) + |x-y|\big] + u|y|
        \right)\\
        &\leq 
    C \frac{e^{-c\frac{|x-y|^2}{u}}}{u^{\frac{d}{2}+1}}
        \left(
        \big[u|y|(|x| + |y|) + \sqrt{u}|y|\big]
        \big[u(|x| + |y|) + \sqrt{u}\big] + u|y|
        \right)\\
        &= C \frac{e^{-c\frac{|x-y|^2}{u}}}{u^{\frac{d}{2}}}
        \left(
        \big[\sqrt{u}|y|(|x| +|y|)+ |y|\big]
        \big[\sqrt{u}(|x| + |y|) + 1\big] + |y|
        \right).
\end{align*}
Let us define the function $m$ by $m(x) = \min\left\{1 , |x|^{-2}\right\}$, $x\in \mathbb{R}^d\setminus\{0\}$, and $m(0)=1$. When $(x,y)\in L_2$ and $0<u<m(x)$, we have that $|y|\leq C(1+|x|)$ and $\sqrt{u}(|x|+|y|)\leq C$. Then, it follows that
\begin{equation*}
   \big[\sqrt{u}|y|(|x|+|y|) + |y|\big]
        \big[\sqrt{u}(|x| + |y|) + 1\big] + |y|\leq C(1+|x|), \quad (x,y)\in L_2\mbox{ and }0<u<m(x).
\end{equation*}
 Thus we obtain that
\begin{equation*}
    |\partial_{x_j} h^{I, B_1}_u (x,y) - 
    \partial_{x_j} h^{I, -I}_u (x,y)|
    \leq C \frac{e^{-c\frac{|x-y|^2}{u}}}{u^{\frac{d}{2}}} (1+ |x|),\quad (x,y)\in L_2\mbox{ and }0<u<m(x).
\end{equation*}

By using this estimation, we get, when $(x,y)\in L_2$,
\begin{equation*}
    \begin{split}
        \int_0^{m(x)} \frac{1}{\sqrt{u}}|\partial_{x_j} h^{I, B_1}_u  (x,y) - \partial_{x_j} h^{I, -I}_u (x,y)| du & \leq 
        C (1+ |x|)\int_0^{m(x)} \frac{e^{-c\frac{|x-y|^2}{u}}}{u^{\frac{d + 1}{2}}} du \leq C \frac{1+|x|}{|x-y|^{d-1}}.
            \end{split}
\end{equation*}

On the other hand, by \eqref{3.5.1} and \eqref{3.5.2} and using Lemma~\ref{Lem2.1} we deduce that
\begin{equation*}
    \begin{split}
        \int_{m(x)}^\infty \frac{1}{\sqrt{u}}
        |\partial_{x_j} h^{I, B_1}_u  (x,y) - \partial_{x_j} h^{I, -I}_u (x,y)| du
        & \leq 
        \int_{m(x)}^\infty \frac{1}{\sqrt{u}}
        |\partial_{x_j} h^{I, B_1}_u  (x,y)| du+ \int_{m(x)}^\infty \frac{1}{\sqrt{u}}
        |\partial_{x_j} h^{I, -I}_u (x,y)| du
                \\ &\hspace{-6cm} \leq C\left(
        \int_{m(x)}^\infty 
        \frac{e^{-\frac{|e^{-u}x-e^{uR}y|^2}{1-e^{-2u}}}}{(1-e^{-2u})^{\frac{d}{2}+1}\sqrt{u}}  |e^{-u}x-e^{uR}y| e^{-u} du+ 
        \int_{m(x)}^\infty 
        \frac{e^{-\frac{|e^{-u}x-y|^2}{1-e^{-2u}}}}{(1-e^{-2u})^{\frac{d}{2}+1}\sqrt{u}}
        |e^{-u}x-y| e^{-u} du\right)
                \\ & \hspace{-6cm}\leq C
        \int_{m(x)}^\infty 
        \frac{e^{-c\frac{|x-y|^2}{1-e^{-2u}}}e^{-u}}{(1-e^{-2u})^{\frac{d+1}{2}}\sqrt{u}}
         du\leq  C
        \int_{m(x)}^\infty 
        \frac{e^{-\frac{|x-y|^2}{u}}}{u^{\frac{d}{2}+1}} du\leq  \frac{C}{\sqrt{m(x)}}
        \int_{m(x)}^\infty 
        \frac{e^{-\frac{|x-y|^2}{u}}}{u^{\frac{d+1}{2}}} du
        \\ & \hspace{-6cm}\leq  \frac{C}{\sqrt{m(x)}|x-y|^{d-1}}
         \leq C \frac{1+|x|}{|x-y|^{d-1}},\quad (x,y)\in L_2.
    \end{split}
\end{equation*}
 We conclude that~\eqref{3.5} holds. Thus we proved~\eqref{3.4}.

According to~\eqref{3.3} and~\eqref{3.4}, by using Lemma~\ref{Lem2.5} we conclude that the operator $D^{E,h}_{k,j,\text{\rm loc}}$ is bounded from $L^p(\mathbb{R}^d,\gamma_{\infty,1})$ into itself, for every $1\leq p \leq \infty$.

We now consider the operators $S^{I,B_1}_{*,k,j,M}$, $G^{I,B_1}_{k,j,M}$ and $\mathcal{V}^{I,B_1}_{\rho,k,j,M}$. Let $E\subset (0,\infty )$ and $h\geq 0$. We define the operator
\begin{equation*}
    \mathcal{H}^{E, h}_{k,j,M}(f)(x,t) =
    \int_{\mathbb{R}^d} H^{E,h}_{k,j,M} (x-y,t) f(y)dy, \quad x\in\mathbb{R}^d \text{ and } t>0,
\end{equation*}
where 
\begin{equation*}
    H^{E,h}_{k,j,M}(z,t) =
    \frac{t^{k+\frac{1}{2}}}{\Gamma(M)} \int_E
    \partial^k_t \left[t^{-M}e^{-\frac{u+h}{t}}
    \right]  \partial_{x_j} W_u(z) (u+h)^{M-1}du,\quad z\in \mathbb{R}^n\mbox{ and }t>0.
\end{equation*}
Minkowski inequality leads to
\begin{equation*}
    \left\| \mathbb{S}^{E,h}_{k,j,M} (f)(x,\cdot) -
    \mathcal{H}^{E,h}_{k,j,M} (f)(x,\cdot)
    \right\|_X \leq \int_{\mathbb{R}^d}
    \left\| \mathfrak{s}^{E,h}_{k,j,M} (x, y, \cdot) -
    H^{E,h}_{k,j,M} (f)(x-y,\cdot)
    \right\|_X |f(y)|dy,\quad x\in \mathbb{R}^d.
\end{equation*}

Our objective is to see that the operator $Z^{E,h}_{k,j,M,\text{loc}}$ defined by
\begin{equation*}
    Z^{E,h}_{k,j,M,\text{loc}} (f)(x) =
    \int_{\mathbb{R}^d}
    \left\| \mathfrak{s}^{E,h}_{k,j,M} (x, y, \cdot) - H^{E,h}_{k,j,M} (f)(x-y,\cdot)
    \right\|_X  \varphi(x,y)|f(y)|dy, \quad x\in\mathbb{R}^d,
\end{equation*}
is bounded from $L^p(\mathbb{R}^d,\gamma_{\infty,1})$ into itself, for every $1\leq p \leq \infty$.

By using again Minkowski inequality we get
\begin{equation*}
    \begin{split}
    \left\| \mathfrak{s}^{E,h}_{k,j,M} (x, y, \cdot) \right.&\left. - H^{E,h}_{k,j,M} (f)(x- y,\cdot)
    \right\|_X  \\ & \hspace{-2cm}\leq 
    \frac{1}{\Gamma(M)} \int_E \left\| 
    t^{k+\frac{1}{2}} \partial^k_t \left[t^{-M}
    e^{-\frac{u+h}{t}}
    \right] \right\|_X 
    |\partial_{x_j}h^{I,B_1}_u(x,y) - \partial_{x_j}W_u(x-y)|(u+h)^{M-1}du,\quad x,y\in \mathbb{R}^d.
    \end{split}
\end{equation*}
We are going to see that
\begin{equation}\label{3.6}
\left\|t^{k+\frac{1}{2}}\partial_t^k\big[t^{-M}e^{-\frac{u+h}{t}}\big]\right\|_X\leq C(u+h)^{\frac{1}{2}-M},\quad u>0.
\end{equation}
We firstly consider $k=0$. We have that
$$
\left\|t^{\frac{1}{2}-M}e^{-\frac{u+h}{t}}\right\|_{L^\infty ((0,\infty ),dt)}\leq C(u+h)^{\frac{1}{2}-M},\quad u>0,
$$
and
$$
\left\|t^{\frac{1}{2}-M}e^{-\frac{u+h}{t}}\right\|_{L^2 ((0,\infty ),\frac{dt}{t})}=\left(\int_0^\infty t^{-2M}e^{-2\frac{u+h}{t}}dt\right)^{\frac{1}{2}}\leq C(u+h)^{\frac{1}{2}-M},\quad u>0,
$$
provided that $M>\frac{1}{2}$.

We also get
$$
\left\|t^{\frac{1}{2}-M}e^{-\frac{u+h}{t}}\right\|_{V_\rho (0,\infty )}\leq \int_0^\infty \Big|\partial _t\big[t^{\frac{1}{2}-M}e^{-\frac{u+h}{t}}\big]\Big|dt\leq C\int_0^\infty t^{-M-\frac{1}{2}}e^{-c\frac{u+h}{t}}dt\leq C(u+h)^{\frac{1}{2}-M},\quad u>0,
$$
when $M>\frac{1}{2}$.

Suppose now $k\geq 1$. We can write 
$$
\partial _t^k=\sum_{i=0}^{k-1}c_iw^{2k-i}\partial _w^{k-i},\quad w=\frac{1}{t}\in (0,\infty ),
$$
where $c_i\in \mathbb{R}$, $i=0,...,k-1$. It follows that
\begin{equation}\label{3.6.1}
\partial_t^k\big[t^{-M}e^{-\frac{u+h}{t}}\big]=\sum_{i=0}^{k-1}\sum_{\ell =0}^{k-i}c_{i,\ell}w^{2k+M-\ell -i}e^{-(u+h)w}(u+h)^{k-i-\ell}, \quad w=\frac{1}{t},\quad u\in (0,\infty ),
\end{equation}
being $c_{i,\ell}\in \mathbb{R}$, $i=0,...,k-1$, $\ell =0,...,k-i$. Here $c_{i,\ell}=0$, when $M<\ell \leq k-i$, $i=0,...,k-1$.

We have that
\begin{align*}
\left\|t^{k+\frac{1}{2}}\partial_t^k\big[t^{-M}e^{-\frac{u+h}{t}}\big]\right\|_{L^\infty ((0,\infty ),dt)}&\leq C\sum_{i=0}^{k-1}\sum_{\ell =0}^{k-i}(u+h)^{k-i-\ell}\sup_{w\in (0,\infty )}w^{k+M-\ell -i-\frac{1}{2}}e^{-(u+h)w}\\
&\leq C(u+h)^{\frac{1}{2}-M},\quad u>0,
\end{align*}
when $M>\frac{1}{2}$.

We also obtain
\begin{align*}
\left\|t^{k+\frac{1}{2}}\partial_t^k\big[t^{-M}e^{-\frac{u+h}{t}}\big]\right\|_{L^2((0,\infty ),\frac{dt}{t})}&\leq C\sum_{i=0}^{k-1}\sum_{\ell =0}^{k-i}(u+h)^{k-i-\ell}\left(\int_0^\infty w^{2k+2M-2\ell -2i-2}e^{-(u+h)w}dw\right)^{\frac{1}{2}}\\
&\leq C(u+h)^{\frac{1}{2}-M},\quad u>0,
\end{align*}
and
\begin{align*}
\left\|t^{k+\frac{1}{2}}\partial_t^k\big[t^{-M}e^{-\frac{u+h}{t}}\big]\right\|_{V_\rho (0,\infty )}&\leq \int_0^\infty \Big|\partial _t\Big[t^{k+\frac{1}{2}}\partial_t^k\big[t^{-M}e^{-\frac{u+h}{t}}\big]\Big]\Big|dt\\
&\leq C\left(\int_0^\infty \Big|t^{k-\frac{1}{2}}\partial_t^k\big[t^{-M}e^{-\frac{u+h}{t}}\big]\Big|dt+\int_0^\infty \Big|t^{k+\frac{1}{2}}\partial_t^{k+1}\big[t^{-M}e^{-\frac{u+h}{t}}\big]\Big|dt\right)\\
&\leq C\left(\sum_{i=0}^{k-1}\sum_{\ell =0}^{k-i}(u+h)^{k-i-\ell}\int_0^\infty w^{k+M-\ell -i-\frac{3}{2}}e^{-(u+h)w}dw\right.\\
&\left.\quad +\sum_{i=0}^k\sum_{\ell =0}^{k+1-i}(u+h)^{k+1-i-\ell}\int_0^\infty w^{k+M-\ell -i-\frac{1}{2}}e^{-(u+h)w}dw\right)\\
&\leq C(u+h)^{\frac{1}{2}-M},\quad u>0,
\end{align*}
provided that $M>\frac{1}{2}$. Thus \eqref{3.6} is established.

From \eqref{3.6} we deduce that
$$
\left\|\mathfrak{s}_{k,j,M}^{E,h}(x,y,\cdot )-H_{k,j,M}^{E,h}(x-y,\cdot )\right\|_X\leq C\int_0^\infty \frac{1}{\sqrt{u}}|\partial _{x_j}h_u^{I,B_1}(x,y)-\partial _{x_j}W_u(x-y)|du,\quad x,y\in \mathbb{R}^d.
$$
By using \eqref{3.4} we get
$$
\left\|\mathfrak{s}_{k,j,M}^{E,h}(x,y,\cdot )-H_{k,j,M}^{E,h}(x-y,\cdot )\right\|_X\leq C\frac{1+|x|}{|x-y|^{d-1}},\quad (x,y)\in L_2,
$$
and then, by virtue of Lemma \ref{Lem2.5}, the operator $Z^{E,h}_{k,j,M,{\rm loc}}$ is bounded from $L^p(\mathbb{R}^d,\gamma_{\infty ,1})$ into itself, for every $1\leq p\leq \infty$.

\subsection{Proof of Claim 4}  Let $E\subset(0,\infty)$ and $h\geq 0$.  By using Minkowski inequality we obtain
$$
\left\|\mathfrak{s}_{k,j}^{E,h}(x,y,\cdot )\right\|_X\leq C\int_E\Big\|t^{k+1}\partial _t^k\big[te^{-\frac{t^2}{4(u+h}}\big]\Big\|_X|\partial _{x_j}h_u^{I,B_1}(x,y)|\frac{du}{(u+h)^{\frac{3}{2}}},\quad x,y\in \mathbb{R}^d.
$$
It was proved (see \eqref{3.2'}, \eqref{3.2.1} and \eqref{3.2.2}) that
$$
\|t^{k+1}\partial _t^k\big[te^{-\frac{t^2}{4(u+h}}\big]\Big\|_X\leq C(u+h),\quad u>0.
$$
Then,
$$
\left\|\mathfrak{s}_{k,j}^{E,h}(x,y,\cdot )\right\|_X\leq C\int_E|\partial _{x_j}h_u^{I,B_1}(x,y)|\frac{du}{\sqrt{u+h}},\quad x,y\in \mathbb{R}^d.
$$
In a similar way by using \eqref{3.6} we can see that
$$
\left\|\mathfrak{s}_{k,j,M}^{E,h}(x,y,\cdot )\right\|_X\leq C\int_E|\partial _{x_j}h_u^{I,B_1}(x,y)|\frac{du}{\sqrt{u+h}},\quad x,y\in \mathbb{R}^d.
$$

We denote now by $E_1$ the sets $(0,t_0)$ and $J_D^\#$ associated with $\delta_1=1/10$ given in Lemma \ref{Lem2.2}. The arguments developed in the proof of \cite[Proposition 4.7]{MN2} allow us to prove that
$$
\int_E|\partial _{x_j}h_u^{I,B_1}(x,y)|\frac{du}{\sqrt{u+h}}\leq Ce^{|x|^2-|y|^2}\min \big\{(1+|x|)^d,(|x|\sin \theta (x,y))^{-d}\big\},\quad (x,y)\in G_1.
$$
We define
$$
\widetilde{\mathbb{V}}_{k,j,M}^{E,h}(f)(x)=\int_{\mathbb{R}^d}\Big\|\mathfrak{s}_{k,j,M}^{E,h}(x,y,\cdot )\Big\|_Xf(y)dy,\quad x\in \mathbb{R}^d.
$$

From Lemma \ref{Lem2.4} we deduce that the operators $\widetilde{V}_{k,j,{\rm glob}}^{E,h}$ and $\widetilde{\mathbb{V}}_{k,j,{\rm glob}}^{E,h}$ are bounded from $L^1(\mathbb{R}^d,\gamma_{\infty ,1})$ into $L^{1,\infty }(\mathbb{R}^d,\gamma_{\infty ,1})$.

Let now $1<p<\infty$ and $\delta \in (\max\{1-\frac{1}{d},1-\frac{1}{p}\},1)$. Observe that $\delta _p$ in Claim 4 is a particular case of such $\delta$. Consider $E_p$ one of the sets $(0,t_0)$ or $J_D^\#$ given in Lemma \ref{Lem2.2} which is associated to $\delta$. By considering the above estimations it is sufficient to see that the operator $Z_{\rm glob}^{E_p,h}$ is bounded from $L^p(\mathbb{R}^d,\gamma_{\infty ,1})$ into itself, where
$$
Z^{E_p,h}(f)(x)=\int_{\mathbb{R}^d}Z^{E_p,h}(x,y)f(y)dy,\quad x\in \mathbb{R}^d,
$$
being
$$
Z^{E_p,h}(x,y)=\int_{E_p}|\partial_{x_j}h_u^{I,B_1}(x,y)|\frac{du}{\sqrt{u+h}},\quad x,y\in \mathbb{R}^d.
$$

We observe that
\begin{align*}
\Big\|Z^{E_p,h}_{\rm glob}(f)\Big\|_{L^p(\mathbb{R}^d,\gamma_{\infty ,1})}^p&=\pi^{-\frac{d}{2}}\int_{\mathbb R^d}|Z^{E_p,h}_{\rm glob}(f)(x)|^pe^{-|x|^2}dx \\
& =\pi^{-\frac{d}{2}}\int_{\mathbb R^d}\left(\int_{\mathbb R^d} Z^{E_p,h}(x,y)(1-\varphi(x,y))e^{-\frac{1}{p}(|x|^2-|y|^2)}f(y)e^{-\frac{|y|^2}{p}}dy\right)^p dx\\
&=\Big\|\widetilde{Z}^{E_p,h}_{\rm glob}(\pi ^{-\frac{d}{2}}e^{-\frac{|y|^2}{p}}f)\Big\|_{L^p(\mathbb{R}^d,dx)}^p,
\end{align*}
where 
$$
\widetilde{Z}^{E_p,h}(g)(x)=\int_{\mathbb R^d}Z^{E_p,h}(x,y)e^{-\frac{1}{p}(|x|^2-|y|^2)}g(y)dy,\quad x\in\mathbb R^d.
$$
Then, we only have to establish that the operator $\widetilde{Z}_{\rm glob}^{E_p,h}$ is bounded from $L^p(\mathbb R^d, dx)$ into itself. For that, we first prove the following estimations.

\noindent (a) For every $\eta\in (0,1)$, 
\begin{equation}\label{(a)}
Z^{E_p,h}(x,y)\leq Ce^{\eta(1-\delta)|x|^2-\eta |y|^2},\quad (x,y)\in G_1,\;\langle x,y\rangle <0.
\end{equation}

\noindent (b) For every $\eta \in (0,1)$ such that $\eta \delta>1-\frac{1}{d}$,
\begin{equation}\label{(b)}
Z^{E_p,h}(x,y)\leq C|x+y|^de^{\eta (1-\frac{\delta}{2})(|x|^2-|y|^2)-\frac{\eta \delta}{2}|x+y||x-y|},\quad (x,y)\in G_1,\;\langle x,y\rangle \geq 0.
\end{equation}

Let $\eta \in (0,1)$. First we observe that from (\ref{3.5.1}) we deduce that
$$
|\partial_{x_j}h_u^{I,B_1}(x,y)|\leq Ce^{-u}\frac{|e^{-u}x-e^{uR}y|}{(1-e^{-2u})^{\frac{d}{2}+1}}e^{-\frac{|e^{-u}x-e^{uR}y|^2}{1-e^{-2u}}},\quad x,y\in \mathbb{R}^d\mbox{ and } u>0.
$$
Then,
$$
|\partial_{x_j}h_u^{I,B_1}(x,y)| \leq Ce^{-u}\frac{e^{-\eta\frac{|e^{-u}x-e^{uR}y|^2}{1-e^{-2u}}}}{(1-e^{-2u})^{\frac{d+1}{2}}} \leq C\frac{e^{-u}}{\sqrt{1-e^{-2u}}}h_u^{I,B_1}(\sqrt{\eta}x,\sqrt{\eta}y),\quad x,y\in \mathbb{R}^d\mbox{ and }u>0.
$$

By using Lemma \ref{Lem2.2} we have that
$$
|\partial_{x_j}h_{\tau(s)}^{I,B_1}(x,y)|\leq C\frac{1-s}{s^{\frac{d+1}{2}}}e^{\eta(|x|^2-|y|^2-\frac{\delta}{4s}Q_s(x,y))}\quad x,y\in \mathbb{R}^d\mbox{ and }s\in\tau^{-1}(E_p).
$$
Since 
$$
\frac{1}{4s}Q_s(x,y)=\frac{s}{4}|x+y|^2+\frac{1}{4s}|x-y|^2+\frac{|x|^2-|y|^2}{2}=\frac{1}{4s}Q_s(y,x)+|x|^2-|y|^2, \quad x,y\in\mathbb R^d \mbox{ and }s\in(0,1), 
$$
and 
$$
\frac{|e^{-u}x-y|^2}{1-e^{-2u}}=\frac{1}{4s}Q_s(y,x),\quad u=\log\frac{1+s}{1-s},\;x,y\in\mathbb R^d\mbox{ and }s\in(0,1),
$$ 
we get
$$
|x|^2-|y|^2-\frac{\delta}{4s}Q_s(x,y)=(1-\delta)(|x|^2-|y|^2)-\delta\frac{|e^{-u}x-y|^2}{1-e^{-2u}},\quad u=\log\frac{1+s}{1-s}\in (0,\infty),\;x,y\in\mathbb R^d.
$$
Thus we obtain
$$
|\partial_{x_j}h_u^{I,B_1}(x,y)|\leq C\frac{e^{-u}}{(1-e^{-2u})^{\frac{d+1}{2}}}e^{\eta(1-\delta)(|x|^2-|y|^2)}e^{-\eta\delta\frac{|e^{-u}x-y|^2}{1-e^{-2u}}}\quad x,y\in \mathbb{R}^d\;\mbox{and}\;u\in E_p.
$$
By making the change of variables $t=1-e^{-2u}$, $u\in(0,\infty)$, it follows that
\begin{align*}
Z^{E_p,h}(x,y)& \leq Ce^{\eta(1-\delta)(|x|^2-|y|^2)} \int_0^1\frac{e^{-\eta\delta  u(t)}}{t^{\frac{d+1}{2}}\sqrt{|\log(1-t)|}}\frac{dt}{\sqrt{1-t}}\\
&\leq Ce^{\eta(1-\delta)(|x|^2-|y|^2)}\int_0^1\frac{e^{-\eta\delta u(t)}}{t^{ \frac{d}{2}+1}}\frac{dt}{\sqrt{1-t}},\quad x,y\in\mathbb R^d.
\end{align*}
Here $u(t)=|\sqrt{1-t}x-y|^2/t$, $x,y\in\mathbb R^d$ and $t\in(0,1)$.

Then, as in the proof of \cite[Theorem 4.2]{P1} we deduce that \eqref{(a)} is satisfied. 

Assume now that $\langle x,y\rangle \geq0$ and $\eta \delta >1-1/d$. As in \cite{P1} we define $t_0=2\frac{\sqrt{a^2-b^2}}{a+\sqrt{a^2-b^2}}$, being $a=|x|^2+|y|^2$ and $b=2\langle x,y\rangle$, and $u_0=\frac{|y|^2-|x|^2+|x+y||x-y|}{2}$. By using \cite[Lemma 4.1]{P1} we obtain
\begin{align*}
\int_0^1\frac{e^{-\eta\delta u(t)}}{t^{\frac{d}{2}+1}\sqrt{1-t}}dt &\leq C\left(\frac{e^{-u_0}}{t_0^{d/2}}\right)^{1-\frac{1}{d}}\int_0^1\frac{e^{-(\eta\delta -1+\frac{1}{d})u(t)}}{t^{ \frac{3}{2}}\sqrt{1-t}}dt \\
&\leq C\left(\frac{e^{-u_0}}{t_0^{d/2}}\right)^{1-\frac{1}{d}}\frac{e^{-(\eta\delta -1+\frac{1}{d})u_0}}{t_0^{ \frac{1}{2}}} \leq C \frac{e^{-\eta\delta u_0}}{t_0^{d/2}},\quad (x,y)\in G_1.
\end{align*}
Then,
$$
Z^{E_p,h}(x,y)\leq C\frac{e^{\eta(1-\delta)(|x|^2-|y|^2)-\eta\delta u_0}}{t_0^{ d/2}},\quad (x,y)\in G_1\mbox.$$
Since $t_0\sim\frac{\sqrt{a^2-b^2}}{a}=\frac{|x+y||x-y|}{a}$, $|x-y||x+y|\geq 1$, $(x,y)\in G_1$ we obtain
$$
Z^{E_p,h}(x,y)\leq C|x+y|^de^{\eta(1-\frac{\delta}{2})(|x|^2-|y|^2)-\frac{\eta\delta}{2}|x-y||x+y|},\quad (x,y)\in G_1,
$$
and \eqref{(b)} is proved.

We now choose $\eta$ such that 
$$
\max\Big\{\frac{1}{p},\frac{1}{\delta}\big(1-\frac{1}{d}\big)\Big\}<\eta <1.
$$
Note that our hypothesis on $\delta$ leads to $0<\frac{1}{\delta}(1-\frac{1}{d})<1$ and also $\eta(1-\delta)<\frac{1}{p}$.

By using \eqref{(a)} and \eqref{(b)} we can deduce that $\widetilde{Z}^{Ep,h}_{\rm glob}$ is a bounded operator from $L^q(\mathbb{R}^d,dx)$ into itself, for every $1\leq q\leq \infty$. Indeed, from \eqref{(a)} we have that
$$
\int_{\{y\in\mathbb R^d:\langle x,y\rangle <0\}}Z^{E_p,h}(x,y)e^{-\frac{1}{p}(|x|^2-|y|^2)}(1-\varphi(x,y))dy \leq C e^{-(\frac{1}{p}-\eta(1-\delta))|x|^2}\int_{\mathbb R^d}e^{-(\eta-\frac{1}{p})|y|^2}dy,\quad x\in\mathbb R^d.
$$
Since $\frac{1}{p}<\eta<1$ and $\eta(1-\delta)<\frac{1}{p}$, it follows that
$$
\sup_{x\in\mathbb R^d}\int_{\{y\in\mathbb R^d:\langle x,y\rangle <0\}}Z^{E_p,h}(x,y)e^{-\frac{1}{p}(|x|^2-|y|^2)}(1-\varphi(x,y))dy<\infty.$$
On the other hand, from \eqref{(b)} we can write
$$
Z^{E_p,h}(x,y)e^{-\frac{1}{p}(|x|^2-|y|^2)}\leq C |x+y|^de^{-\left(\frac{\eta\delta}{2}-\left|\frac{1}{p}-\eta(1-\frac{\delta}{2})\right|\right)|x+y||x-y|},\quad (x,y)\in G_1\mbox{ and }\langle x,y\rangle\geq 0.
$$
By taking into account that $\frac{1}{p}<\eta<1$ and $\eta(1-\delta )<\frac{1}{ p}$, we get that  $\frac{\eta\delta}{2}-\left|\frac{1}{p}-\eta(1-\frac{\delta}{2})\right|>0$. By proceeding as in \cite[p. 501]{P1} we obtain
$$
\sup_{x\in\mathbb R^d}\int_{\{y\in\mathbb R^d:\langle x,y\rangle \geq 0\}}Z^{E_p,h}(x,y)e^{-\frac{1}{p}(|x|^2-|y|^2)}(1-\varphi(x,y))dy<\infty.
$$

We conclude that
$$
\sup_{x\in\mathbb R^d}\int_{\mathbb R^d}Z^{E_p,h}(x,y)e^{-\frac{1}{p}(|x|^2-|y|^2)}(1-\varphi(x,y))dy<\infty.
$$

In a similar way we get
$$
\sup_{y\in\mathbb R^d}\int_{\mathbb R^d}Z^{E_p,h}(x,y)e^{-\frac{1}{p}(|x|^2-|y|^2)}(1-\varphi(x,y))dx<\infty.
$$
We deduce that the operator  $\widetilde{Z}_{\rm glob}^{E_p,h}$ is bounded from $L^q(\mathbb R^d,dx)$ into itself, for every $1\leq q\leq\infty$. Thus we prove that the operator  $Z_{\rm glob}^{E_p,h}$ is bounded from $L^p(\mathbb R^d,\gamma_{\infty,1})$ into itself, and the proof of the Claim 4 is finished.

\section{Proof of Claim 3}
In this section we prove the Claim 3. Assume that $E \subset (0,\infty )$ and $h\geq 0$.

\subsection{Maximal operators}

\subsubsection{}\label{S4.1.1} We consider firstly the local maximal operator $\widetilde{U}^{E,h}_{*,k,j,{\rm loc}}$ defined by
$$
\widetilde{U}^{E,h}_{*,k,j,{\rm loc}}(f)(x)=\sup_{t>0}\left|\int_{\mathbb{R}^d}f(y)\mathfrak{U}_{k,j}^{E,h}(x-y,t)\varphi (x,y)dy\right|,\quad x\in \mathbb{R}^d,
$$
where
$$
\mathfrak{U}_{k,j}^{E,h}(z,t)=\frac{t^{k+1}}{2\sqrt{\pi}}\int_E \partial _t^k\big[te^{-\frac{t^2}{4(u+h)}}\big]\partial _{x_j}W_u(z)\frac{du}{(u+h)^{\frac{3}{2}}},\quad z\in \mathbb{R}^d\mbox{ and }t>0.
$$
We recall that 
$$
W_u(z)=\frac{e^{-\frac{|z|^2}{2u}}}{(2\pi u)^{\frac{d}{2}}},\quad z\in \mathbb{R}^d\mbox{ and }u>0.
$$
Assume that $f\in C_c^\infty (\mathbb{R}^d)$. We can write, for each $x,y\in \mathbb{R}^d$ and $t>0$,
\begin{align*}
    \int_{\mathbb{R}^d}f(y)\mathfrak{U}_{k,j}^{E,h}(x-y,t)\varphi (x,y)dy&=\frac{t^{k+1}}{2\sqrt{\pi}}\int_E \partial _t^k\big[te^{-\frac{t^2}{4(u+h)}}\big]\int_{\mathbb{R}^d}\partial _{x_j}W_u(x-y)f(y)\varphi (x,y)\frac{dydu}{(u+h)^{\frac{3}{2}}}\\
    &=-\frac{t^{k+1}}{2\sqrt{\pi}}\int_E \partial _t^k\big[te^{-\frac{t^2}{4(u+h)}}\big]\int_{\mathbb{R}^d}\frac{x_j-y_j}{u}W_u(x-y)f(y)\varphi (x,y)\frac{dydu}{(u+h)^{\frac{3}{2}}}.
\end{align*}
Then, by using Minkowski inequality and \eqref{3.2'} we get that
\begin{align*}
\widetilde{U}_{*,k,j,{\rm loc}}^{E,h}(f)(x)&\leq C\int_E\|t^{k+1}\partial _t^k[te^{-\frac{t^2}{4(u+h)}}]\|_{L^\infty ((0,\infty ),dt)}
\int_{\mathbb{R}^d}|x-y|W_u(x-y)|f(y)|\varphi (x,y)\frac{dydu}{u(u+h)^{\frac{3}{2}}}\\
&\leq C\int_E \frac{1}{u\sqrt{u+h}}\int_{\mathbb{R}^d}|x-y|W_u(x-y)|f(y)|\varphi (x,y)dydu,\quad x\in \mathbb{R}^d.
\end{align*}

If $h>0$, by taking into account that $|x-y|\leq C$, when $(x,y)\in L_2$, and that $\frac{|z|}{\sqrt{u}}W_u(z)\leq CW_{2u}(z)$, $z\in \mathbb{R}^d$, it follows that
\begin{align*}
  \widetilde{U}_{*,k,j,{\rm loc}}^{E,h}(f)(x) &\leq C\left(
    \int_1^\infty\frac{1}{u\sqrt{u+h}}\int_{\mathbb{R}^d}W_u(x-y)|f(y)|dydu\right.\\
    &\left.\quad +\int_0^1\frac{1}{\sqrt{u(u+h)}}\int_{\mathbb{R}^d}W_{2u}(x-y)|f(y)|dydu\right)\\
    &\leq C\left(\int_1^\infty u^{-\frac{3}{2}}du+\int_0^1u^{-\frac{1}{2}}du\right)\sup_{v>0}W_v(|f|)(x),\quad x\in \mathbb{R}^d.
\end{align*}
On the other hand, we have that, for every $x\in\mathbb{R}^d$ and $t>0$, 
$$
 \widetilde{U}_{*,k,j,{\rm loc}}^{E,0}(f)(x)\leq C\sup_{t>0}\left(t^{k+1}\int_0^\infty \frac{e^{-\frac{t^2}{8u}}}{u^{\frac{k+3}{2}}}du\right)\sup_{v>0}W_v(|f|)(x)\leq C\sup_{v>0}W_v(|f|)(x).
$$
We conclude that
$$
\widetilde{U}_{*,k,j,{\rm loc}}^{E ,h}(f)(x)\leq C\sup_{t>0}W_t(|f|)(x),\quad x\in \mathbb{R}^n.
$$
From wellknown $L^p$-boundedness properties of the maximal operator defined by the classical heat semigroup we deduce that the local maximal operator $\widetilde{U}_{*,k,j,{\rm loc}}^{E ,h}$ is bounded from $L^p(\mathbb{R}^d,dx)$ into itself, for every $1<p<\infty$ and from $L^1(\mathbb{R}^d, dx)$ into $L^{1,\infty }(\mathbb{R}^d,dx)$. 

\subsubsection{}\label{S4.1.2} We now study the local maximal operator
$$
\widetilde{\mathbb{U}}_{*,k,j,M,{\rm loc}}^{E ,h}(f)(x)=\sup_{t>0}\left|\int_{\mathbb{R}^d}\mathfrak{U}_{k,j,M}^{E ,h}(x-y,t)f(y)\varphi (x,y)dy\right|,\quad x\in \mathbb{R}^d,
$$
where 
$$
\mathfrak{U}_{k,j,M}^{E ,h}(z,t)=\frac{t^{k+\frac{1}{2}}}{\Gamma (M)}\int_E \partial _t^k\big[t^{-M}e^{-\frac{u+h}{t}}\big]\partial _{x_j}W_u(z)(u+h)^{M-1}du,\quad z\in \mathbb{R}^d\mbox{ and }t>0.
$$
According to \eqref{3.6} and Minkowski inequality we can write
\begin{align*}
\widetilde{\mathbb{U}}_{*,k,j,M, {\rm loc}}^{E,h}(f)(x)&\leq C\int_E\|t^{k+\frac{1}{2}}\partial _t^k[t^{-M}e^{-\frac{u+h}{t}}]\|_{L^\infty ((0,\infty ),dt)}\frac{|x-y|}{u}W_u(x-y)|f(y)|\varphi (x,y)(u+h)^{M-1}dydu\\
&\leq C\int_0^\infty  \frac{1}{u\sqrt{u+h}}\int_{\mathbb{R}^d}|x-y|W_u(x-y)|f(y)|\varphi (x,y)dydu,\quad x\in \mathbb{R}^d.
\end{align*}
By proceeding as in section \ref{S4.1.1} we get that, if $h>0$,
$$
\widetilde{\mathbb{U}}_{*,k,j,M, {\rm loc}}^{E,h}(f)(x)\leq C\sup_{v>0}W_v(|f|)(x),\quad x\in \mathbb{R}^d.
$$

On the other hand, by using \eqref{3.6.1}, if $k\geq 1$ we get
$$
t^{k+\frac{1}{2}}\int_0^\infty \Big|\partial_t^k\big[t^{-M}e^{-\frac{u}{t}}\big]\Big|u^{M-\frac{3}{2}}du
\leq Ct^{k+\frac{1}{2}}\sum_{i=0}^{k-1}\sum_{\ell =0}^{k-i}t^{-2k-M+\ell +i}\int_0^\infty e^{-\frac{u}{t}}u^{k-i-\ell +M-\frac{3}{2}}du\leq C,\quad t>0.
$$
Furthermore, we obtain
\begin{equation}\label{4.3}
t^{\frac{1}{2}-M}\int_0^\infty e^{-\frac{u}{t}}u^{M-\frac{3}{2}}du\leq C,\quad t>0.
\end{equation}
Then, 
\begin{align*}
\widetilde{\mathbb{U}}_{*,k,j,M, {\rm loc}}^{E,0}(f)(x)&\leq Ct^{k+\frac{1}{2}}\int_0^\infty \partial_t^k\big[t^{-M}e^{-\frac{u}{t}}\big]\Big|u^{M-\frac{3}{2}}\int_{\mathbb{R}^d}\frac{|x-y|}{\sqrt{u}}W_u(x-y)|f(y)|dydu\\
&\leq C\sup_{v>0}W_v(|f|)(x),\quad x\in \mathbb{R}^d.
\end{align*}
We conclude that
$$
\widetilde{\mathbb{U}}_{*,k,j,M,{\rm loc}}^{E ,h}(f)(x)\leq C\sup_{t>0}W_t(|f|)(x),\quad x\in \mathbb{R}^d,
$$
and thus, we establish that $\widetilde{\mathbb{U}}_{*,k,j,M,{\rm loc}}^{E ,h}$ is bounded from $L^p(\mathbb{R}^d,dx)$ into itself, for every $1<p<\infty$, and from $L^1(\mathbb{R}^d,dx)$ into $L^{1,\infty }(\mathbb{R}^d,dx)$.
\vspace{1.5cm}

\subsection{Littlewood-Paley functions}

\subsubsection{}\label{S4.2.1} We consider the local Littlewood-Paley function $\mathfrak{g}_{k,j,{\rm loc}}^{E, h}$ defined by
$$
\mathfrak{g}_{k,j,{\rm loc}}^{E, h}(f)(x)=\left(\int_0^\infty \left|\int_{\mathbb{R}^d}f(y)\mathfrak{U}_{k,j}^{E ,h}(x-y,t)\varphi (x,y)dy\right|^2\frac{dt}{t}\right)^{\frac{1}{2}},\quad x\in \mathbb{R}^d.
$$
By using Minkowski inequality and \eqref{3.2.1} we get
\begin{align*}
    \mathfrak{g}_{k,j, {\rm loc}}^{E, h}(f)(x)&\leq C\int_E \Big\|t^{k+1}\partial _t^k\big[te^{-\frac{t^2}{4(u+h)}}\big]\Big\|_{L^2((0,\infty),\frac{dt}{t})}\int_{\mathbb{R}^d}|x_j-y_j|W_u(x-y)|f(y)|\varphi (x,y)\frac{dydu}{u(u+h)^{\frac{3}{2}}}\\
    &\leq C \int_0^\infty  \frac{1}{u\sqrt{u+h}}\int_{\mathbb{R}^d}|x-y|W_u(x-y)|f(y)|\varphi (x,y)dydu\leq C\sup_{v>0}W_v(|f|)(x),\quad x\in \mathbb{R}^d.
\end{align*}
As in section \ref{S4.1.1} we deduce that if $h>0$ the operator $\mathfrak{g}_{k,j,{\rm loc}}^{E, h}$ is bounded from $L^p(\mathbb{R}^d,dx)$ into itself, for every $1<p<\infty$, and from $L^1(\mathbb{R}^d,dx)$ into $L^{1,\infty }(\mathbb{R}^d,dx)$.

We now study the operator $\mathfrak{g}_{k,j,{\rm loc}}^{E, 0}$. We consider the Littlewood-Paley function $\mathfrak{g}_{k,j}^{E, 0}$ defined by
$$
\mathfrak{g}_{k,j}^{E, 0}(f)(x)=\left(\int_0^\infty \left|\int_{\mathbb{R}^d}f(y)\mathfrak{U}_{k,j}^{E ,0}(x-y,t)dy\right|^2\frac{dt}{t}\right)^{\frac{1}{2}},\quad x\in \mathbb{R}^d.
$$
We are going to see that $\mathfrak{g}_{k,j}^{E, 0}$ is bounded from $L^2(\mathbb{R}^d,dx)$ into itself. If $\mathcal{F}$ denotes the Fourier transform defined in $L^1(\mathbb{R}^d)$ by
$$
\mathcal{F}(f)(z)=\int_{\mathbb{R}^d}g(x)e^{-i\langle x,z\rangle}dx,\quad z\in \mathbb{R}^d,
$$
we have that (see for instance, \cite[p. 15 (11)]{EMOT1})
\begin{equation}\label{4.2'}
\mathcal{F}(W_t)(x)= e^{-\frac{t|x|^2}{2}},\quad x\in \mathbb{R}^d\mbox{ and }t>0.
\end{equation}
By using Plancherel equality and \eqref{4.2'} we get
\begin{align*}
    \int_{\mathbb{R}^d}|\mathfrak{g}_{k,j}^{E ,0}(f)(x)|^2dx&=\int_0^\infty \int_{\mathbb{R}^d}\left|\int_{\mathbb{R}^d}f(y)\mathfrak{U}_{k,j}^{E, 0}(x-y,t)dy\right|^2dx\frac{dt}{t}\\
    &  =\int_0^\infty \int_{\mathbb{R}^d}\left|\mathcal F\left(\int_{\mathbb{R}^d}f(y)\mathfrak{U}_{k,j}^{E, 0}(\cdot-y,t)dy\right)(z)\right|^2dz\frac{dt}{t}\\
    & =\int_0^\infty \int_{\mathbb{R}^d}\Big|\int_E\frac{t^{k+1}}{2\sqrt{\pi}} \frac{\partial _t^k\big[te^{-\frac{t^2}{4u}}\big]}{u^{\frac{3}{2}}}\int_{\mathbb R^d} f(y)\mathcal{F}(\partial _{x_j}W_u(\cdot -y))(z)dydu\Big|^2dz\frac{dt}{t}\\
    &=\int_0^\infty \int_{\mathbb{R}^d}\Big|\mathcal{F}(f)(z)\frac{t^{k+1}}{2\sqrt{\pi}}\int_E \frac{\partial _t^k\big[te^{-\frac{t^2}{4u}}\big]}{u^{\frac{3}{2}}}\mathcal{F}(\partial _{x_j}W_u)(z)du\Big|^2dz\frac{dt}{t}\\
    &= \int_{\mathbb{R}^d}|\mathcal{F}(f)(z)|^2\int_0^\infty \Big|\frac{t^{k+1}}{2\sqrt{\pi}}\int_E \frac{\partial _t^k\big[te^{-\frac{t^2}{4u}}\big]}{u^{\frac{3}{2}}}z_je^{-\frac{u|z|^2}{2}}du\Big|^2\frac{dt}{t}dz,\quad f\in C_c^\infty (\mathbb{R}^d).
\end{align*}
Minkowski inequality and \eqref{3.2.1} leads to 
\begin{align*}
    \int_{\mathbb{R}^d}|\mathfrak{g}_{k,j}^{E ,0}(f)(x)|^2dx&\leq C\int_{\mathbb{R}^d}|\mathcal{F}(f)(z)|^2\left(\int_0^\infty |z|e^{-\frac{u|z|^2}{2}}\big\|t^{k+1}\partial _t^k\big[te^{-\frac{t^2}{4u}}\big]\big\|_{L^2(0,\infty),\frac{dt}{t})}\frac{du}{u^{\frac{3}{2}}}\right)^2dz\\
    &\leq C\int_{\mathbb{R}^d}|\mathcal{F}(f)(z)|^2\left(\int_0^\infty \frac{|z|}{\sqrt{u}}e^{-\frac{u|z|^2}{2}}du\right)^2dz\\
    &\leq C\int_{\mathbb{R}^d}|\mathcal{F}(f)(z)|^2dz=C\|f\|_{L^2(\mathbb{R}^d,dx)}^2,\quad f\in C_c^\infty (\mathbb{R}^d).
\end{align*}
We now use the Banach-valued Calder\'on-Zygmund theory for singular integrals (see \cite{RbFRT}). We recall that the operator $U_{k,j}^{E ,0}$ is defined by
$$
U_{k,j}^{E ,0}(f)(x,t)=\int_{\mathbb{R}^d}\mathfrak{U}_{k,j}^{E, 0}(x-y,t)f(y)dy,\quad x\in \mathbb{R}^d\mbox{ and }t>0.
$$
It is clear that
$$
\mathfrak{g}_{k,j}^{E ,0}(f)(x)=\left\|U_{k,j}^{E, 0}(f)(x,\cdot )\right\|_{L^2((0,\infty ),\frac{dt}{t})},\quad x\in \mathbb{R}^d.
$$
 By using again Minkowski inequality and \eqref{3.2.1} 
 \begin{align}\label{4.3'}
 \left\|\mathfrak{U}_{k,j}^{E, 0}(x-y,\cdot)\right\|_{L^2((0,\infty ),\frac{dt}{t})}&\leq C\int_0^\infty \Big\|t^{k+1}\partial_t^k[te^{-\frac{t^2}{4u}]}\Big\|_{L^2((0,\infty ),\frac{dt}{t})}\frac{|x-y|}{u}\frac{e^{-\frac{|x-y|^2}{2u}}}{u^{\frac{d+3}{2}}}du\nonumber\\
 &\leq C\int_0^\infty \frac{e^{-c\frac{|x-y|^2}{u}}}{u^{\frac{d}{2}+1}}du\leq \frac{C}{|x-y|^d},\quad x,y\in \mathbb{R}^d,\;x\not=y.
 \end{align}
 In a similar way we obtain, for every $i=1,...,d$,
 \begin{equation}\label{4.4}
 \left\|\partial _{x_i}\mathfrak{U}_{k,j}^{E, 0}(x-y,\cdot)\right\|_{L^2((0,\infty ),\frac{dt}{t})}
 + \left\|\partial _{y_i}\mathfrak{U}_{k,j}^{E, 0}(x-y,\cdot)\right\|_{L^2((0,\infty ),\frac{dt}{t})}\leq \frac{C}{|x-y|^{d+1}},\quad x,y\in \mathbb{R}^d,\;x\not=y.
 \end{equation}
 
 Let $N\in \mathbb{N}$, $N\geq 2$. The space $\mathbb{F}_N=L^2((1/N,N),dt/t)$ is a Banach and separable space. 
 
 Assume that $f\in C_c^\infty (\mathbb{R}^d)$. Let $x\in \mathbb{R}^d$. We consider the mapping $F_x:\mathbb{R}^d\longrightarrow \mathbb{F}_N$ defined, for every $y\in \mathbb{R}^d$, by $F_x(y):[\frac{1}{N},N]\longrightarrow \mathbb{C}$ such that 
 $$
 [F_x(y)](t)=f(y) \mathfrak{U}_{k,j}^{E , 0}(x-y,t),\quad t\in \Big[\frac{1}{N},N\Big].
 $$
 We observe that $F_x(y)$, $y\in \mathbb{R}^d$, is continuous in $[1/N,N]$. Thus, $F_x$ is continuous in $\mathbb{R}^d$. Indeed, let $y_0\in \mathbb{R}^d$. We can write, by \eqref{3.2} and Minkowski inequality
 \begin{align*}
     \|F_x(y)-F_x(y_0)\|_{\mathbb{F}_N}&\leq C\Big\|\int_0^\infty \frac{t^{k+1}e^{-\frac{t^2}{8u}}}{u^{(k+2)/2}}|f(y)\partial _{x_j}W_u(x-y)-f(y_0)\partial _{x_j}W_u(x-y_0)|du\Big\|_{\mathbb{F}_N}\\
     &\leq C\Big\|\int_0^\infty te^{-v}v^{k/2-1}|f(y)\partial _{x_j}W_{\frac{t^2}{8v}}(x-y)-f(y_0)\partial _{x_j}W_{\frac{t^2}{8v}}(x-y_0)|dv\Big\|_{\mathbb{F}_N}\\
     &\leq C\int_0^\infty e^{-v}v^{k/2-1}\big\|t(f(y)\partial _{x_j}W_{\frac{t^2}{8v}}(x-y)-f(y_0)\partial _{x_j}W_{\frac{t^2}{8v}}(x-y_0))\big\|_{\mathbb{F}_N}dv,\quad y\in \mathbb{R}^d.
 \end{align*}
 Since
\begin{align*}
 \big\|t(f(y)\partial _{x_j}W_{\frac{t^2}{8v}}(x-y)-f(y_0)\partial _{x_j}W_{\frac{t^2}{8v}}(x-y_0))\big\|_{\mathbb{F}_N}&\\
 &\hspace{-5cm}= (2\pi)^{-d/2}\left(\int_{\frac{1}{N}}^N t\Big(|f(y)(x_j-y_j)e^{-4\frac{|x-y|^2v}{t^2}}-f(y_0)(x_j-(y_0)_j)e^{-4\frac{|x-y_0|^2v}{t^2}}|\Big(\frac{8v}{t^2}\Big)^{{\frac{d}{2}}+1}\Big)^2dt\right)^{\frac{1}{2}}\\
 &\hspace{-5cm}\leq C\left(\int_{\frac{1}{N}}^N\frac{v^{d+2}}{t^{2d+3}}\big(|x-y|^2e^{-8\frac{|x-y|^2v}{t^2}}+|x-y_0|^2e^{-8\frac{|x-y_0|^2v}{t^2}}\big)dt\right)^2\\
 &\hspace{-5cm}\leq Cv^{\frac{d+1}{2}}\int_{\frac{1}{N}}^N\frac{dt}{t^{2d+1}}\leq Cv^{\frac{d+1}{2}},\quad y\in \mathbb{R}^d\mbox{ and }v>0,
 \end{align*}
 by using the dominated convergence theorem we deduce that
 $$
 \lim_{y\rightarrow y_0}\|F_x(y)-F_x(y_0)\|_{\mathbb{F}_N}=0.
 $$
 Since $\mathbb{F}_N$ is a separable Banach space, Pettis' Theorem (\cite[Theorem p. 131]{Yo}) implies that $F_x$ is $\mathbb{F}_N$ -strongly measurable.
 
 By \eqref{4.3'} we get
 $$
 \int_{\mathbb{R}^d}\|\mathfrak{U}_{k,j}^{E, 0}(x-y,\cdot )\|_{\mathbb{F}_N}|f(y)|dy<\infty ,\quad x\not \in {\rm supp }f.
 $$
 We define 
 $$
\widetilde{U}_{k,j}^{E ,0}(f)(x)=\int_{\mathbb{R}^d}\mathfrak{U}_{k,j}^{E ,0}(x-y,\cdot )f(y)dy,\quad x\not\in {\rm supp }f,
 $$
 where the integral is understood in the $\mathbb{F}_N$-Bochner sense. Suppose that $g\in \mathbb{F}_N$. We have that
 $$
 \int_{\frac{1}{N}}^N\int_{\mathbb{R}^d}|\mathfrak{U}_{k,j}^{E,0}(x-y,t)f(y)g(t)|dydt\leq C\|g\|_{\mathbb{F}_N}\int_{\mathbb{R}^d}\frac{|f(y)|}{|x-y|^d}dy<\infty,\quad x\not \in {\rm supp }f.
 $$
 We can write
 $$
 \int_{\frac{1}{N}}^Ng(t)\Big[\widetilde{U}_{k,j}^{E ,0}(f)(x)\Big](t)\frac{dt}{t}=\int_{\mathbb{R}^d}f(y)\int_{\frac{1}{N}}^N \mathfrak{U}_{k,j}^{E ,0}(x-y,t)g(t)\frac{dt}{t}dy=\int_{\frac{1}{N}}^Ng(t)U_{k,j}^{E,0}(f)(x,t)\frac{dt}{t},\quad x\not\in {\rm supp }f.
 $$
 Hence, for every $x\not\in {\rm supp }f$, $\widetilde{U}_{k,j}^{E,0}(f)(x)=U_{k,j}^{E, 0}(f)(x,\cdot)$, in $\mathbb{F}_N$.
 
 We also have that
 \begin{equation}\label{4.5}
 \int_{\mathbb{R}^d}\Big\|\widetilde{U}_{k,j}^{E,0}(f)(x)\Big\|_{\mathbb{F}_N}^2dx\leq \int_{\mathbb{R}^d}|\mathfrak{g}_{k,j}^{E, 0}(f)(x)|^2dx\leq C\|f\|_{L^2(\mathbb{R}^d,dx)}^2.
 \end{equation}
 
 According to \eqref{4.3'}, \eqref{4.4} and \eqref{4.5} and by taking into account that the constant $C$ in \eqref{4.5} does not depend on $N$, the vector-valued Calder\'on-Zygmund theory (\cite{RbFRT}) allows us to conclude that $\widetilde{U}_{k,j}^{E ,0}$ defines a bounded operator from 
 
 (i) $L^p(\mathbb{R}^d,dx)$ into $L_{\mathbb{F}_N}^p(\mathbb{R}^d,dx)$ and  
 $$
 \sup_{\substack{N\in \mathbb{N}\\N\geq 2}}\Big\|\widetilde{U}_{k,j}^{E, 0}\Big\|_{L^p(\mathbb{R}^d,dx)\longrightarrow L^p_{\mathbb{F}_N}(\mathbb{R}^d,dx)}<\infty,
 $$
 for every $1<p<\infty$;
 
 (ii) $L^1(\mathbb{R}^d,dx)$ into $L_{\mathbb{F}_N}^{1,\infty }(\mathbb{R}^d,dx)$ and 
 $$
 \sup_{\substack{N\in \mathbb{N}\\N\geq 2}}\Big\|\widetilde{U}_{k,j}^{E, 0}\Big\|_{L^1(\mathbb{R}^d,dx)\longrightarrow L_{\mathbb{F}_N}^{1,\infty }(\mathbb{R}^d,dx)}<\infty.
 $$
By using monotone convergence theorem we deduce that $\mathfrak{g}_{k,j}^{E, 0}$ is bounded from $L^p(\mathbb{R}^d,dx)$ into itself, for every $1<p<\infty$, and from $L^1(\mathbb{R}^d,dx)$ into $L^{1,\infty}(\mathbb{R}^d, dx)$.

\subsubsection{}\label{S4.2.2} We are going to study the local Littlewood-Paley function $\mathbb{G}_{k,j,M,{\rm loc}}^{E, h}$ defined by
$$
\mathbb{G}_{k,j,M,{\rm loc}}^{E, h}(f)(x)=\left(\int_0^\infty \Big|\int_{\mathbb{R}^d}\mathfrak{U}_{k,j,M}^{E, h}(x-y,t)f(y)\varphi (x,y)dy\Big|^2\frac{dt}{t}\right)^{\frac{1}{2}},\quad x\in \mathbb{R}^d.
$$
Suppose that $h>0$. Minkowski inequality and \eqref{3.6} lead to 
\begin{align*}
    \mathbb{G}_{k,j,M,{\rm loc}}^{E, h}(f)(x)&\leq C\int_{E}\Big\|t^{k+\frac{1}{2}}\partial _t^k\big[t^{-M}e^{-\frac{u+h}{t}}\big]\Big\|_{L^2((0,\infty ),\frac{dt}{t})}\frac{(u+h)^{M-1}}{u}\int_{\mathbb{R}^d}|x_j-y_j|W_u(x-y)|f(y)|\varphi (x,y)dydu\\
    &\leq C\int_E\frac{1}{u\sqrt{u+h}}\int_{\mathbb{R}^d}|x_j-y_j|W_u(x-y)|f(y)|\varphi (x,y)dydu\leq C\sup_{v>0}|W_v(|f|)(x)|,\quad x\in \mathbb{R}^d.
\end{align*}
It follows that $\mathbb{G}_{k,j,M,{\rm loc}}^{E, h}$ is bounded from $L^p(\mathbb{R}^d,dx)$ into itself, for every $1<p<\infty$, and from $L^1(\mathbb{R}^d,dx)$ into $L^{1,\infty}(\mathbb{R}^d, dx)$. 

In order to study the operator $\mathbb{G}_{k,j,M,{\rm loc}}^{E, 0}$ we use the vector-valued Calder\'on-Zygmund theory (\cite{RbFRT}).

We define 
$$
\mathbb{G}_{k,j,M}^{E, 0}(f)(x)=\left(\int_0^\infty \left|\int_{\mathbb{R}^d}\mathfrak{U}_{k,j,M}^{E ,0}(x-y,t)f(y)dy\right|^2\frac{dt}{t}\right)^{\frac{1}{2}},\quad x\in \mathbb{R}^d.
$$
By using Plancherel equality, Minkowski inequality and \eqref{3.6} we obtain
\begin{align*}
    \int_{\mathbb{R}^d}|\mathbb{G}_{k,j,M}^{E, 0}(f)(x)|^2dx&=\int_0^\infty \int_{\mathbb{R}^d}\left|\int_{\mathbb{R}^d}\mathfrak{U}_{k,j,M}^{E ,0}(x-y,t)f(y)dy\right|^2dx\frac{dt}{t}\\
    &=\int_0^\infty \int_{\mathbb{R}^d}\left|\mathcal{F}(f)(z)\frac{t^{k+\frac{1}{2}}}{\Gamma (M)}\int_E \partial _t^k\big[t^{-M}e^{-\frac{u}{t}}\big]\mathcal{F}(\partial_{x_j}W_u)(z)u^{M-1}dy\right|^2dz\frac{dt}{t}\\
    &\leq C\int_{\mathbb{R}^d}|\mathcal{F}(f)(z)|^2\left(\int_0^\infty |z|e^{-u|z|^2}\frac{du}{\sqrt{u}}\right)^2dz\leq C\|f\|_{L^2(\mathbb{R}^d,dx)}^2.
\end{align*}
We consider the operator $\mathbb{U}_{k,j,M}^{E, 0}$ defined by
$$
\mathbb{U}_{k,j,M}^{E,0}(f)(x,t)=\int_{\mathbb{R}^d}\mathfrak{U}_{k,j,M}^{E, 0}(x-y,t)f(y)dy,\quad x\in \mathbb{R}^d\mbox{ and }t>0.
$$
We have that
$$
\mathbb{G}_{k,j,M}^{E, 0}(f)(x)=\left\|\mathbb{U}_{k,j,M}^{E, 0}(f)(x,\cdot )\right\|_{L^2((0,\infty ),\frac{dt}{t})},\quad x\in \mathbb{R}^d.
$$
Minkowski inequality and \eqref{3.6} lead to
$$
\Big\|\mathfrak{U}_{k,j,M}^{E ,0}(x-y,\cdot)\Big\|_{L^2((0,\infty),\frac{dt}{t})}\leq C\int_0^\infty \frac{e^{-\frac{|x-y|^2}{u}}}{u^{\frac{d}{2}+1}}du\leq \frac{C}{|x-y|^d},\quad x,y\in \mathbb{R}^d,\;x\not=y.
$$
We also obtain, for every $i=1,..,d$,
$$
\Big\|\partial_{x_i}\mathfrak{U}_{k,j,M}^{E ,0}(x-y,\cdot)\Big\|_{L^2((0,\infty ),\frac{dt}{t})}+\Big\|\partial_{y_i}\mathfrak{U}_{k,j,M}^{E ,0}(x-y,\cdot)\Big\|_{L^2((0,\infty ),\frac{dt}{t})}\leq \frac{C}{|x-y|^{d+1}},\quad x,y\in \mathbb{R}^d,\;x\not=y.
$$
By proceeding as in the previous case we can conclude that the operator $\mathbb{G}_{k,j,M,{\rm loc}}^{E, 0}$ is bounded from $L^p(\mathbb{R}^d,dx)$ into itself, for every $1<p<\infty$, and from $L^1(\mathbb{R}^d,dx)$ into $L^{1,\infty }(\mathbb{R}^d,dx)$.

\subsection{Variation operators}
\subsubsection{}\label{S4.3.1}
We study the local variation operator defined by
\begin{equation*}
    \mathfrak{v}^{E,h}_{\rho,k,j,\text{loc}}(f)(x) = 
    V_\rho\left(
    t\rightarrow U^{E, h}_{k,j,\text{loc}} (f)(x,t)
    \right),\quad x\in \mathbb{R}^d,
\end{equation*}
where
\begin{equation*}
    U^{E,h}_{k,j}(f)(x,t) 
    = \int_{\mathbb{R}^d} \mathfrak{U}^{E,h}_{k,j} (x-y,t) f(y) dy, \quad x\in\mathbb{R}^d \text{ and } t>0.
\end{equation*}

By using~\eqref{3.2.2} we can proceed as in section \ref{S4.1.1} to prove that, when $h>0$,  $\mathfrak{v}^{E,h}_{\rho,k,j,\text{loc}}$ is bounded from $L^p(\mathbb{R}^d, dx)$ into itself, for every, $1<p<\infty$, and from $L^1(\mathbb{R}^d, dx)$ into $L^{1,\infty}(\mathbb{R}^d, dx)$.

We now study the variation operator $\mathfrak{v}^{E,0}_{\rho,k,j,\text{loc}}$. We consider first, $E = (0,\infty)$. The classical Poisson kernel is given by
\begin{equation*}
    \mathbb{P}_t(z) = \frac{\Gamma(\frac{d+1}{2})}{\pi^{\frac{d+1}{2}}} \frac{t}{(t^2+ |z|^2)^{\frac{d+1}{2}}}, \quad z\in\mathbb{R}^d \text{ and } t>0.
\end{equation*}

We have that 
\begin{equation*}
    \partial^{k+1}_t \left[\mathbb{P}_t(z) \right]
    = \partial^{k+2}_t \left[
    \frac{\Gamma(\frac{d+1}{2})}{\pi^{\frac{d+1}{2}}(1-d)} \frac{1}{(t^2 + |z|^2)^{\frac{d-1}{2}}}
    \right], \quad z\in\mathbb{R}^d \text{ and } t>0.
\end{equation*}
According to Fa\`a di Bruno's formula
\begin{equation*}
    \partial^{k+2}_t \left[\frac{1}{(t^2 + |z|^2)^{\frac{d-1}{2}}}
    \right] = 
    \sum_{0\leq \ell\leq \frac{k+2}{2}} a_\ell t^{k+2-2\ell}
    \frac{1}{(t^2 + |z|^2)^{\frac{d-1}{2}+k+2-\ell}}, \quad z\in\mathbb{R}^d \text{ and } t>0,
\end{equation*}
for certain $a_\ell\in\mathbb{R}$, $0\leq \ell\leq \frac{k+1}{2}$, $\ell\in\mathbb{N}$. Then,
\begin{equation*}
    t^{k+1}\partial^{k+1}_t \mathbb{P}_t(z) = \phi_t(z), \quad  z\in\mathbb{R}^d \text{ and } t>0,
\end{equation*}
where $\phi_t(z) = t^{-d}\phi(z/t)$, $z\in\mathbb{R}^d$ and $t>0$, and
\begin{equation*}
    \phi(z) = \sum_{0\leq \ell\leq \frac{k+2}{2}}  
    \frac{a_\ell}{(1 + |z|^2)^{\frac{d-1}{2}+k+2-\ell}}, \quad  z\in\mathbb{R}^d.
\end{equation*}

On the other hand, given that $\mathbb{P}_t$ is a radial function, the relation between the Fourier and the Hankel Transform, leads to (see \cite[p. 7 (4) and p. 24 (18)]{EMOT2}),
\begin{equation*}
    \mathcal{F}(\partial_{x_j}\partial^k_t \mathbb{P}_t)(z) = 
    i z_j |z|^k e^{-t|z|} = i \frac{z_j}{|z|}|z|^{k+1}e^{-t|z|},
    \quad z\in\mathbb{R}^d\setminus\{0\} \text{ and } t>0.
\end{equation*}

We get
\begin{equation*}
    t^{k+1}\partial_{x_j}\partial^k_t \mathbb{P}_t (f)(x)
    = C t^{k+1} \partial^{k+1}_t \mathbb{P}_t (R_jf)(x)
    = C(\phi_t * R_j f)(x), \quad  x\in\mathbb{R}^d \text{ and } t>0,
\end{equation*}
for a certain $C\in\mathbb{R}$. Here, ${R}_j$ denotes the $j$-th Euclidean Riesz transform. 

We define 
$$\psi(u) = \sum_{0\leq \ell \leq \frac{k+2}{2}} 
\frac{a_\ell}{(1+u^2)^{\frac{d+1}{2}+k+2-\ell}},\quad u\in(0,\infty).
$$  
It is clear that $\phi(z) = \psi(|z|)$, $z\in \mathbb{R}^d$. We have that $\psi(u)\rightarrow 0$ as $u\rightarrow \infty$, and $\int_0^\infty |\psi'(u)|u^d du <\infty$. According to Lemma 2.4 in~\cite{CJRW}, the variation operator associated with $\{T_t\}_{t>0}$ where $T_t f = \phi_t * f$, $t>0$, is bounded from $L^p(\mathbb{R}^d,dx)$ into itself, for every $1<p<\infty$, and from $L^1(\mathbb{R}^d, dx)$ into $L^{1,\infty}(\mathbb{R}^d,dx)$. Since $R_j$ is bounded from $L^p(\mathbb{R}^d,dx)$ into itself, for every $1<p<\infty$ we derive the same boundedness property for $\mathfrak{v}^{(0,\infty), 0}_{\rho,k,j}$.

In order to see that $\mathfrak{v}^{(0,\infty), 0}_{\rho,k,j}$ is bounded from $L^1(\mathbb{R}^d, dx)$ into $L^{1,\infty}(\mathbb{R}^d,dx)$, we use vector valued Calder\'on-Zygmund theory. 

According to~\eqref{3.2.2} and Minkoswki inequality we obtain
\begin{equation}\label{4.1}
    \begin{split}
        \Big\|\mathfrak{U}^{(0,\infty),0}_{k,j}(x-y,\cdot)\Big\|_{V_\rho} & \leq C
        \int_0^\infty \frac{\|t^{k+1}\partial^k_t [te^{\frac{-t^2}{4u}}]\|_{V_\rho}}{u^{\frac{3}{2}}} \frac{|x_j - y_j|}{u^{\frac{d}{2} +1}}
        e^{\frac{-|x-y|^2}{4u}} du
        \\ & \leq C
        \int_0^\infty \frac{e^{\frac{-c|x-y|^2}{u}}}{u^{\frac{d}{2} +1}} du \leq \frac{C}{|x-y|^d},\quad x, y\in \mathbb{R}^d,\;x\neq y,
            \end{split}
\end{equation}
and, in a similar way, for every $i=1,...,d$,
\begin{equation}\label{4.2}
\Big\|\partial_{x_i}\mathfrak{U}^{(0,\infty),0}_{k,j}(x-y,\cdot)\Big\|_{V_\rho}+\Big\|\partial_{y_i}\mathfrak{U}^{(0,\infty),0}_{k,j}(x-y,\cdot)\Big\|_{V_\rho} \leq \frac{C}{|x-y|^{d+1}},\quad x, y\in \mathbb{R}^d,\;x\neq y.
\end{equation}

Let $N\in\mathbb{N}$, $N\geq 2$. We consider the space $V_\rho \left(\left[\frac{1}{N},N\right]\right)$ consisting in all $g\in C\left(\left[\frac{1}{N},N\right]\right)$ such that
\begin{equation*}
    V^N_\rho(g) = \sup_{\substack{ \frac{1}{N}< t_\ell < t_{\ell-1} < \dots < t_1 <N\\ \ell\in\mathbb{N}}}
    \left(\sum_{i = 1}^{\ell-1} |g(t_{i+1}) - g(t_i)|^\rho \right)^{\frac{1}{\rho}} < \infty.
\end{equation*}

By identifying those functions that differ in a constant $\left(V_\rho \left(\left[\frac{1}{N},N\right]\right),V^N_\rho \right)$ is a Banach space.
Assume that $f\in C_c^\infty(\mathbb{R}^d)$. Let $x\in \mathbb{R}^d$. We define, for every $y\in \mathbb{R}^d$, the function $F_x(y)\in C\left(\left[\frac{1}{N},N\right]\right)$ defined by
\begin{equation*}
    [F_x(y)](t) = f(y) \mathfrak{U}^{(0,\infty),0}_{k,j}(x-y,t),\quad  t\in \left[\tfrac{1}{N},N\right].
\end{equation*}

By \cite[Lemma~4]{BCCFR} we get
\begin{equation*}
    \begin{split}
    V^N_\rho(F_x(y)) & \leq C
    \int_{\frac{1}{N}}^N \int_0^\infty 
    \left| \partial_t \left[
    t^{k+1} \partial^k_t\left( 
    t e^{-\frac{t^2}{4u}}
    \right)\right]\right|
    \frac{|x_j - y_j|}{u^{\frac{d}{2}+ \frac{5}{2}}}
    e^{-\frac{|x-y|^2}{2u}} du dt
    \\ & \leq C 
    \int_{\frac{1}{N}}^N \int_0^\infty 
    \left( \frac{t^{k+1}}{u^{\frac{k}{2}}}
    +  \frac{t^{k}}{u^{\frac{k-1}{2}}}   \right)
    \frac{e^{-\frac{t^2}{8u}}}{u^{\frac{d}{2}+ 2}}
    e^{-c\frac{|x-y|^2}{u}} du dt
    \\ & \leq C
    \int_0^\infty 
    \frac{e^{-c\frac{N^{-2}+|x-y|^2}{u}}}{u^{\frac{d+3}{2}}} du
    \\ & \leq 
    \frac{C}{(N^{-2}+|x-y|^2)^{d+1}}, \quad y\in\mathbb{R}^d.
    \end{split}
\end{equation*}
We define the mapping 
\begin{equation*}
    \begin{split}
        F_x : \mathbb{R}^d & \longrightarrow V_\rho\left(\left[\tfrac{1}{N},N\right]\right)
        \\ y & \longrightarrow F_x(y).
    \end{split}
\end{equation*}

$F_x$ is continuous. Indeed, let $y_0\in \mathbb{R}^d$. By Lemma~4 in~\cite{BCCFR} we obtain
\begin{equation*}
    \begin{split}
        V^N_\rho & (F_x(y) - F_x(y_0)) 
        \\ & \leq C
        \int_{\frac{1}{N}}^N \int_0^\infty 
        \frac{\left| \partial_t \left[
        t^{k+1} \partial^k_t\left( 
        t e^{-\frac{t^2}{4u}}
        \right)\right]\right|}{u^{\frac{3}{2}}}
        \left| \frac{x_j-y_j}{u} W_u(x-y) f(y)
        - \frac{x_j - (y_0)_j}{u} W_u(x-y_0) f(y_0)
        \right| du dt
        \\ & \leq C
        \int_{\frac{1}{N}}^N \int_0^\infty 
        \frac{t^k e^{-c\frac{t^2}{u}}}{u^{\frac{k+2}{2}}}
        \left| \frac{x_j-y_j}{u} W_u(x-y) f(y)
        - \frac{x_j - (y_0)_j}{u} W_u(x-y_0) f(y_0)
        \right| du dt
        \\ & \leq C
        \int_{\frac{1}{N}}^N \int_0^\infty 
        \frac{e^{-cv}v^{\frac{k}{2}}}{t^2}
        \left| (x_j-y_j) W_{\frac{t^2}{v}}(x-y) f(y)
        - (x_j - (y_0)_j) W_{\frac{t^2}{v}}(x-y_0) f(y_0)
        \right| dv dt.
    \end{split}
\end{equation*}

Since
\begin{equation*}
    \int_{\frac{1}{N}}^N \frac{1}{t^2}
    \left| (x_j-y_j) W_{\frac{t^2}{v}}(x-y) f(y)
        - (x_j - y_{0,j}) W_{\frac{t^2}{v}}(x-y_0) f(y_0)
        \right|dt \leq 
        C v^{{\frac{d-1}{2}}},
        \quad y\in\mathbb{R}^d\text{ and } v>0,
\end{equation*}
the dominated convergence theorem leads to 
\begin{equation*}
    \lim_{y\rightarrow y_0} V_{\rho}^N (F_x(y) - F_x(y_0)) = 0.
\end{equation*}

Since $F_x$ is continuous, it is $V_\rho\left(\left[\tfrac{1}{N},N\right]\right)$-strongly measurable by Pettis' Theorem. Indeed, $F_x$ is weakly measurable. Furthermore, if $\mathbb{Q}$ represents the set of rational numbers, we have that $\displaystyle F_x(\mathbb{R}^d)= \overline{F_x(\mathbb{Q}^d)}^{V_\rho\left(\left[\tfrac{1}{N},N\right]\right)}$.

By using~\eqref{4.1} we get
\begin{equation*}
    \int_{\mathbb{R}^d} \left\|
    \mathfrak{U}^{(0,\infty),0}_{k,j}(x-y,\cdot)
    \right\|_{V^N_\rho} |f(y)| dy 
    \leq 
    \int_{\mathbb{R}^d} \frac{|f(y)|}{|x-y|^{d}} dy
    < \infty,\quad x\notin \text{supp}(f).
\end{equation*}

We define
\begin{equation*}
    \widetilde{M}^{(0,\infty),0}_{k,j} (f)(x) 
    = \int_{\mathbb{R}^d} \mathfrak{U}^{(0,\infty),0}_{k,j}(x-y,\cdot) f(y) dy,\quad x\notin \text{supp}(f),
\end{equation*}
where the integral is understood in the $V_\rho$- Bochner sense.

Let $0<a\neq 1$. We define the functional $T_a$ in $V_\rho\left(\left[\tfrac{1}{N},N\right]\right)$ as follows
\begin{equation*}
    T_a(g) = g(a) - g(1).
\end{equation*}

It is clear that $T_a \in V_\rho\left(\left[\tfrac{1}{N},N\right]\right)'$, the dual space of $V_\rho\left(\left[\tfrac{1}{N},N\right]\right)$. We have that
\begin{equation*}
    \begin{split}
        T_a\left( \widetilde{M}^{(0,\infty),0}_{k,j} (f)(x) 
        \right) 
        & = 
        \int_{\mathbb{R}^d} T_a \left(
        \mathfrak{U}^{(0,\infty),0}_{k,j}(x-y,\cdot) 
        \right) f(y) dy
        \\ & = U^{(0,\infty),0}_{k,j} (f)(x,a)
        - U^{(0,\infty),0}_{k,j} (f)(x,1),
        \quad x\notin \text{supp}(f).
    \end{split}
\end{equation*}

We deduce that, for every $x\notin \text{supp}(f)$, there exists $c_x\in\mathbb{R}$ such that 
\begin{equation*}
    \left[ \widetilde{M}^{(0,\infty),0}_{k,j} (f)(x) 
    \right](t) = U^{(0,\infty),0}_{k,j} (f)(x,t) 
    + c_x, \quad t\in (0,\infty),
\end{equation*}
and
\begin{equation*}
    \widetilde{M}^{(0,\infty),0}_{k,j} (f)(x) = U^{(0,\infty),0}_{k,j} (f)(x,\cdot)
\end{equation*}
as elements of $V_\rho\left(\left[\tfrac{1}{N},N\right]\right)$.

Since $\mathfrak{v}^{(0,\infty),0}_{\rho,k,j}$ is bounded from $L^p(\mathbb{R}^d,dx)$ into itself, for every $1<p<\infty$, from~\eqref{4.1} and~\eqref{4.2}we deduce that $\widetilde{M}^{(0,\infty),0}_{k,j}$ is bounded from $L^1(\mathbb{R}^d,dx)$ into
$L_{V_\rho([\tfrac{1}{N},N])}^{1,\infty}\left(\mathbb{R}^d,dx\right)$. Furthermore 
\begin{equation*}
    \sup_{N\in\mathbb{N},N\geq 2} \|\widetilde{M}^{(0,\infty),0}_{k,j}\|_{
    L^1(\mathbb{R}^d,dx)\hookrightarrow 
    L_{V_\rho([1/N, N])}^{1,\infty}(\mathbb{R}^d,dx)
    } <\infty .
\end{equation*}
By using the monotone convergence theorem we conclude that $\mathfrak{v}^{(0,\infty),0}_{\rho,k,j}$ is bounded from 
$L^1(\mathbb{R}^d,dx)$ into $L^{1,\infty}(\mathbb{R}^d,dx)$.

Suppose now that $E \subset (\eta,\infty)$ for some $\eta>0$. By using Lemma~4 in~\cite{BCCFR} we get
\begin{equation*}
    \begin{split}
        \| \mathfrak{U}^{E,0}_{k,j} (x-y,\cdot) \|_{V_\rho}
        & \leq C \int_{E} \int_0^{\infty} 
         \frac{\left| \partial_t \left[
        t^{k+1} \partial^k_t\left( 
        t e^{-\frac{t^2}{4u}}
        \right)\right]\right|}{u^{\frac{3}{2}}}
        \frac{|x_j-y_j|}{u} W_u(x-y) dt du
        \\ & \leq C 
         \int_\eta^\infty \int_0^{\infty} 
         \frac{t^k e^{-c\frac{t^2}{u}}}{u^{\frac{k+4}{2}}} dt
         \, \frac{e^{-c\frac{|x-y|^2}{u}}}{u^{\frac{d-1}{2}}} du
         \\ & \leq \frac{C}{|x-y|^{d-1}}
         \int_\eta^\infty \frac{du}{u^{\frac{3}{2}}} 
         \int_0^{\infty} v^{\frac{k-1}{2}}e^{-v}dv
         \\ & \leq C
         \frac{1 + |x|}{|x-y|^{d-1}},
         \quad  x\neq y.
    \end{split}
\end{equation*}

According to Lemma~\ref{Lem2.5}, the operator $\mathfrak{v}^{E,0}_{\rho,k,j,\text{loc}}$ is bounded from 
 $L^p(\mathbb{R}^d,dx)$ into itself, for every $1<p<\infty$, and from $L^1(\mathbb{R}^d,dx)$ into $L^{1,\infty}(\mathbb{R}^d,dx)$. 
 
 By combining the above results we deduce that if $E \subset(0,\eta)$, for some $\eta>0$, the operator $\mathfrak{v}^{E,0}_{\rho,k,j,\text{loc}}$ is bounded from
 $L^p(\mathbb{R}^d,dx)$ into itself, for every $1<p<\infty$, and from $L^1(\mathbb{R}^d,dx)$ into $L^{1,\infty}(\mathbb{R}^d,dx)$. 

\subsubsection{}
We consider the local variation operator defined by
\begin{equation*}
    \mathfrak v^{E, h}_{\rho,k,j,M,\text{loc}} f(x)
    = V_\rho \left(
    t \rightarrow U^{E, h}_{k,j,M,\text{loc}}(f)(x,t)
    \right), \quad x\in\mathbb{R}^d,
\end{equation*}
where
\begin{equation*}
    U^{E, h}_{k,j,M}(f)(x,t) 
    = \int_{\mathbb{R}^d} f(y) \mathfrak{U}^{E, h}_{k,j,M}
    (x-y,t) dy, \quad x\in\mathbb{R}^d \text{ and }t>0.
\end{equation*}

Minkowski inequality and~\eqref{3.6} allow us, as in Section  \ref{S4.1.2}, to see that the operator  $\mathfrak v^{E, h}_{\rho,k,j,M,\text{loc}}$, when $h>0$, is bounded from
 $L^p(\mathbb{R}^d,dx)$ into itself, for every $1<p<\infty$, and from $L^1(\mathbb{R}^d,dx)$ into $L^{1,\infty}(\mathbb{R}^d,dx)$. 
 
 We are going to study the variation operator $ \mathfrak v^{E, 0}_{\rho,k,j,M,\text{loc}}$. In order to do this we consider firstly the operator $ \mathfrak v^{(0,\infty), 0}_{\rho,k,j,M}$. We recall that
 $$
 \mathfrak U_{k,j,M}^{(0,\infty),0}(z,t)=\frac{t^{k+1/2}}{\Gamma(M)}\partial^k_t\partial_{x_j}\int_0^\infty t^{-M}e^{-u/t}W_u(z)u^{M-1}du,\quad z\in\mathbb R^d\mbox{ and } t>0. 
 $$
 
 We can write
 $$
 \mathfrak U_{k,j,M}^{(0,\infty),0}(z,t)=L_{k,j,M}(z,t),\quad z\in\mathbb R^d\mbox{ and } t>0,
 $$
 where
 $$
 L_{k,j,M}(z,t)=\frac{t^{k+1/2}}{\Gamma(M)}\partial^k_t\partial_{x_j}\int_0^\infty e^{-v}W_{vt}(z)v^{M-1}dv,\quad z\in\mathbb R^d\mbox{ and } t>0.
 $$
 We have that $\mathcal F(\partial_{z_j}W_{vt}(z))(w)=iw_je^{-\frac{|w|^2tv}{2}}$, $w\in\mathbb R^d$ and $t, v>0$. Then, 
 $$
 \mathcal F(t^{k+1/2}\partial^k_t\partial_{z_j}W_{vt}(z))(w)=\frac{(-1)^k}{2^k}iw_jt^{k+1/2}v^k|w|^{2k}e^{-\frac{tv|w|^2}{2}}, \quad w\in\mathbb R^d\mbox{ and }t,v>0,
 $$
 and we get
 \begin{align*}
 \mathcal{F}(U_{k,j,M}^{(0,\infty ),0}(\cdot,t))(w)&=\mathcal{F}(f*L_{k,j,M}(\cdot ,t))(w)=-i\frac{w_j}{|w|}\mathcal{F}(f)(w)\Psi_{k,t}(w)\\
 &=\mathcal{F}(R_jf)(w)\mathcal{F}(\mathcal{F}^{-1}(\Psi_{k,t}))(w),\quad w\in \mathbb{R}^d \mbox{ and }t>0,
 \end{align*}
 where $R_j$ denotes the $j$-th Euclidean Riesz transformation and
 $$
 \Psi_{k,t}(w)=\frac{(-1)^{k+1}}{2^k\Gamma (M)}t^{k+\frac{1}{2}}\int_0^\infty e^{-v} v^{k+M-1}|w|^{2k+1}e^{-\frac{tv|w|^2}{2}}dv,\quad w\in \mathbb{R}^d\mbox{ and }t>0.
 $$
  We consider the function $\psi_k(w)=|w|^{2k+1}e^{-a|w|^2}$, $w\in\mathbb R^d$, where $a>0$. We define the Hankel transform $h_\nu$ by
 $$
 h_\nu(g)(s)=\int_0^\infty(sr)^{-\nu}J_\nu(sr)g(r)r^{2\nu +1}d\nu,\quad s\in (0,\infty),
 $$
 where $J_\nu$ denotes the Bessel function of the first kind and order $\nu >-1$. Since $\psi_k$ is a radial function we obtain
 $$
 \mathcal F^{-1}(\psi_k)(x)=\frac{1}{(2\pi )^d}h_{\frac{d-2}{2}}(\widetilde{\psi}_k)(|x|),\quad x\in\mathbb R^d,
 $$
 where $\widetilde{\psi}_k(r)=r^{2k+1}e^{-ar^2}$, $r\in (0,\infty)$. According to \cite[(14), p. 30]{EMOT2} we get
 $$
 \mathcal{F}^{-1}(\psi_k)(s)=c_ka^{-\frac{d+1}{2}-k}\;_1F_1\Big(\frac{d+1}{2}+k,\frac{d}{2};-\frac{s^2}{4a}\Big),\quad s\in (0,\infty),
 $$
 for certain $c_k>0$. Here $_1F_1$ represents the Kummer's confluent hypergeometric function (see, for example, \cite[\S 9]{Leb}).

Thus,
\begin{align*}
\mathcal F^{-1}(\Psi_{k,t})(x)&=\frac{(-1)^{k+1}2^{\frac{d+1}{2}}c_k}{\Gamma (M)}t^{-\frac{d}{2}}\int_0^\infty e^{-v} v^{M-1-\frac{d+1}{2}}\;_1F_1\Big(\frac{d+1}{2}+k,\frac{d}{2};-\frac{|x|^2}{2tv}\Big)dv,\quad x\in \mathbb{R}^d,
\end{align*}
and we can write
$$
U_{k,j,M}^{(0,\infty ),0}(f)(x,t)=\big((\phi_k)_{\sqrt{t}}*R_jf\big)(x),\quad x\in \mathbb{R}^d\mbox{ and }t>0,
$$
where
$$
\phi _k(x)=\frac{(-1)^{k+1}2^{\frac{d+1}{2}}c_k}{\Gamma (M)}\int_0^\infty e^{-v}v^{M-1-\frac{d+1}{2}}\;_1F_1\Big(\frac{d+1}{2}+k,\frac{d}{2};-\frac{|x|^2}{2v}\Big)dv,\quad x\in \mathbb{R}^d.
$$
By taking into account \cite[(9.11.2)]{Leb} with the adequate correct form (see also \cite[Theorem 1]{Pa}), for every $\alpha,\gamma\in \mathbb{R}$, with $\gamma \not=0,-1,-2,...$, we have that $\;_1F_1(\alpha,\gamma;z)=e^z\;_1F_1(\gamma-\alpha,\gamma; -z)$, $z\in \mathbb{R}$. Then,
$$
\lim_{z\rightarrow +\infty}\;_1F_1\Big(\frac{d+1}{2}+k,\frac{d}{2};-z\Big)=0.
$$
Since $M>\frac{d+1}{2}$, we infer by the dominated convergence Theorem that $\lim_{|x|\rightarrow +\infty}\phi_k(x)=0$.

On the other hand let us consider the function 
$$
\widetilde{\phi}_k(r)=\frac{(-1)^{k+1}2^{\frac{d+1}{2}}c_k}{\Gamma (M)}\int_0^\infty e^{-v}v^{M-1-\frac{d+1}{2}}\;_1F_1\Big(\frac{d+1}{2}+k,\frac{d}{2};-\frac{r^2}{2v}\Big)dv,\quad r\in (0,\infty).
$$
It is clear that $\phi_k(x)=\widetilde{\phi}_k(|x|)$, $x\in \mathbb{R}^d$.

Again by \cite[(9.11.2)]{Leb} we deduce that
\begin{align*}
    \int_0^\infty \big|\frac{d}{dr}\widetilde{\phi}_k(r)\big|r^ddr&\leq C\int_0^\infty e^{-v}v^{M-2-\frac{d+1}{2}}\int_0^\infty r^{d+1}\Big|\;_1F_1\Big(\frac{d+1}{2}+k+1,\frac{d}{2}+1;-\frac{r ^2}{2v}\Big)\Big|drdv\\
    &\leq C\int_0^\infty e^{-v}v^{M-2-\frac{d+1}{2}}\left(\int_0^{\sqrt{v}}+\int_{\sqrt{v}}^\infty \right)r^{d+1}
    \Big|\;_1F_1\Big(\frac{d+1}{2}+k+1,\frac{d}{2}+1;-\frac{r^2}{2v}\Big)\Big|drdv\\
    &\leq C\left(\int_0^\infty e^{-v}v^{M-\frac{3}{2}}dv+\int_0^\infty e^{-v}v^{M-2-\frac{d+1}{2}}\int_{\sqrt{v}}^\infty r^{-2k-2}v^{\frac{d+1}{2}+k+1}dr dv\right)\\
    &=C\int_0^\infty e^{-v}v^{M-\frac{3}{2}}dv<\infty.
\end{align*}

From \cite[Lemma 2.4]{CJRW} we deduce that the variation operator $V_\rho (t\rightarrow (\phi_k)_{\sqrt{t}}*f)$ is bounded from $L^p(\mathbb{R}^d,dx)$ into itself, for every $1<p<\infty$, and from $L^1(\mathbb{R}^d,dx)$ into $L^{1,\infty }(\mathbb{R}^d,dx)$. Since $R_j$ is bounded from $L^p(\mathbb{R}^d,dx)$ into itself, for every $1<p<\infty$, we deduce that $U_{\rho ,k,j,M}^{(0,\infty ),0}$ is bounded from $L^p(\mathbb{R}^d,dx)$ into itself, for every $1<p<\infty$.

In order to prove that $\mathfrak v_{\rho ,k,j,M}^{(0,\infty ),0}$ is bounded from $L^1(\mathbb{R}^d,dx)$ into $L^{1,\infty }(\mathbb{R}^d,dx)$ we can use the vector-valued Calder\'on-Zygmund theory by using the techniques developed to prove the corresponding property for the operator $\mathfrak v_{\rho,k,j,{\rm loc}}^{(0,\infty),0}$ in section \ref{S4.3.1}. 

Suppose that $E\subset (\eta,\infty )$ for some $\eta>0$. From \eqref{3.6} it follows that 
$$
\left\|\mathfrak{U}_{k,j,M}^{E, 0}(x,y,\cdot )\right\|_{V_\rho}\leq C\int_E |\partial_{x_j}W_u(x-y)|\frac{du}{\sqrt{u}}\leq \frac{C}{|x-y|^{d-1}}\int_\eta ^\infty \frac{du}{u^{\frac{3}{2}}}\leq C\frac{1+|x|}{|x-y|^{d-1}},\quad x,y\in \mathbb R^d,\;x\neq y.
$$
By using Lemma \ref{Lem2.5} we prove that $\mathfrak v_{\rho ,k,j,M,{\rm loc}}^{E, 0}$ is bounded from $L^p(\mathbb{R}^d,dx)$ into itself, for every $1<p<\infty$, and from $L^1(\mathbb{R}^d,dx)$ into $L^{1,\infty }(\mathbb{R}^d,dx)$.

The above properties allow us to conclude that $\mathfrak v_{\rho ,k,j,M,{\rm loc}}^{E, 0}$ is bounded from $L^p(\mathbb{R}^d,dx)$ into itself, for every $1<p<\infty$, and from $L^1(\mathbb{R}^d,dx)$ into $L^{1,\infty }(\mathbb{R}^d,dx)$ provided that $E \subset (0,\eta)$, for some $\eta >0$.

\begin {rem}
Note that the properties proved in this section for maximal operators and Littlewood-Paley functions hold for every subset $E$ of $(0,\infty )$. However, we prove the properties for variation operators when $E\subset (0,\eta)$ or $E\subset (\eta, \infty )$, for some $\eta >0$.
\end{rem}

\end{document}